\documentclass[12pt,reqno,sort]{elsarticle}

\usepackage{amsfonts,color,amsmath,amssymb,fancyhdr}
\usepackage{amsthm}
\usepackage{graphicx}
\usepackage{tabularx}
\usepackage{soul}
\usepackage{color, xcolor}
\usepackage{indentfirst}
\usepackage{comment}

\usepackage{etoolbox}

\usepackage{float}
\usepackage{booktabs}
\usepackage{graphicx}
\usepackage{tabularx}

\usepackage[colorlinks,linkcolor=blue, anchorcolor=red,citecolor=green]{hyperref}
\usepackage{enumerate,enumitem}
\usepackage{fullpage}
\usepackage{multirow}
\usepackage{makecell,array}

\usepackage{comment}

\usepackage{adjustbox}

\usepackage{graphicx,anysize,float,booktabs,geometry}
\usepackage{caption} 
\usepackage{hyperref}

\hypersetup{citecolor=red, linkcolor=blue, colorlinks=true}

\def\Box{\vcenter{\vbox{\hrule\hbox{\vrule
     \vbox to 8.8pt{\hbox to 10pt{}\vfill}\vrule}\hrule}}}

\allowdisplaybreaks

\newcommand{\mc}{\mathcal}

\newcommand{\cL}{\mathcal{L}}

\newcommand{\cP}{\mathcal{P}}

\newcommand\Aut{{\rm Aut}}

\newcommand{\Fix}{{\rm Fix}}

\newcommand\PSU{{\rm PSU}}

\newcommand\PSL{{\rm PSL}}

\newcommand\GL{{\rm GL}}

\newcommand\PSp{{\rm PSp}}
\newcommand\Sp{{\rm Sp}}

\newcommand{\PA}{\mathcal{P}}
\newcommand{\LA}{\mathcal{L}}
\newcommand{\Ta}{T_{\alpha}}

\theoremstyle{plain}
\newtheorem{theorem}{Theorem}[section]

\newtheorem{lemma}[theorem]{Lemma}

\numberwithin{equation}{section}

\theoremstyle{remark}

\def\<{\langle}
\def\>{\rangle}

\def\a{\alpha}

\begin{document}

\newcommand{\stopthm}{\begin{flushright}
		\(\box \;\;\;\;\;\;\;\;\;\; \)
\end{flushright}}
\newcommand{\symfont}{\fam \mathfam}

\title{Generalized quadrangles with a point-primitive and line-primitive automorphism group with socle $\PSp_4(q)$} 

\date{}
\author[add1]{Meizi Ou}\ead{meiziou@zju.edu.cn}
\author[add2]{Jianbing Lu \corref{cor1}}\ead{jianbinglu@nudt.edu.cn}\cortext[cor1]{Corresponding author}

\address[add1]{School of Mathematical Sciences, Zhejiang University, Hangzhou 310058,   China}
\address[add2]{College of Science, National University of Defense Technology, Changsha 410073, China}

\begin{abstract}
  Let $\mc{S}$ be a finite thick generalized quadrangle, and let $G\leq \Aut(\mc{S})$ act primitively on both  points and   lines. Building on the almost simple reduction for point-primitive and line-primitive actions, we study the case where the socle of $G$ is the projective symplectic group $\PSp_4(q)$ with $q\ge 3$. We show that, up to duality, either $\mc{S}\cong W(3,q)$ for $q\geq 3$, or $q=3$ and $\mc{S}\cong H(3,4)$.
\newline 

	\noindent\text{Keywords:} Generalized quadrangle; point-primitive; line-primitive.
	
\end{abstract}	

\maketitle

\section{Introduction}

Generalized quadrangles were introduced by Tits \cite{tits1959trialite} as a geometric tool for studying groups of Lie type.  More generally, a \emph{generalized $n$-gon} \cite{van2012generalized} is a point--line incidence geometry whose incidence graph has diameter $n$ and girth $2n$. A finite generalized $n$-gon has order $(s,t)$ if each line is incident with $s+1$ points and each point is incident with $t+1$ lines; it is \emph{thick} if $s,t>1$. By the Feit--Higman Theorem~\cite{feit1964nonexistence}, finite thick generalized $n$-gons exist only for $n\in\{3,4,6,8\}$. In this paper we focus on the case $n=4$, namely generalized quadrangles. 

A finite \emph{generalized quadrangle} is also defined as an incidence structure $\mc{S}=(\mathcal P,\mathcal L,\mathcal I)$ with disjoint point set $\mathcal P$ and line set $\mathcal L$ such that: (1) any two distinct points are incident with at most one common line; (2) for any point $x$ and any line $\ell$ not incident with $x$, there is a unique point $y$ on $\ell$ that is collinear with $x$. An \emph{automorphism} of $\mathcal{S}$ is a permutation of the point set that preserves collinearity. Identifying each line with the set of points incident with it, an automorphism of $\mathcal{S}$ naturally induces a permutation of the line set that preserves concurrency.  All automorphisms of $\mathcal{S}$ form a group, which is called the full automorphism group and denoted by $\Aut(\mc{S})$. An automorphism group $G$ of $\mc{S}$ is a subgroup of $\Aut(\mc{S})$. 

The main examples of generalized quadrangles are the classical generalized quadrangles, which arise from classical simple groups of Lie type of rank $2$. Among these, the symplectic quadrangle $W(3,q)$ is defined from a $4$-dimensional symplectic vector space $V$ over $\mathbb F_q$, where the points are the totally isotropic $1$-dimensional subspaces of $V$ and the lines are the totally isotropic $2$-dimensional subspaces, with incidence given by inclusion. The automorphism group of $W(3,q)$ is almost simple with socle $\PSp_4(q)$, the projective symplectic group.  Other classical generalized quadrangles, such as the parabolic quadric $Q(4,q)$ and the Hermitian quadrangle $H(3,q^2)$, can be defined analogously from orthogonal and unitary spaces, respectively.
 
Symmetry conditions on generalized quadrangles are an important tool for understanding their structure. A major open problem is Kantor's conjecture~\cite{kantor1990automorphism}, which states that a finite flag-transitive generalized quadrangle is classical, $\mathrm{GQ}(3,5)$ or the generalized quadrangle of order (15,17) arising from the Lunelli-Sce hyperoval up to duality. In \cite{buekenhout1994finite}, Buekenhout and Van Maldeghem proved that a finite thick generalized polygon with a point-distance-transitive automorphism group is classical, dual classical, or the unique generalized quadrangle of order $(3,5)$. More recently, Bamberg, Li and Swartz classified antiflag-transitive generalized quadrangles~\cite{bamberg2018classification} and locally $2$-transitive generalized quadrangles~\cite{bamberg2021classification}, giving further support for Kantor's conjecture. In this paper, we focus on the problem of classifying finite thick generalized quadrangles that admit an automorphism group acting primitively on both points and lines. We note that the automorphism groups of all classical generalized quadrangles act primitively on both points and lines. A key reduction due to Bamberg et al.~\cite{bamberg2012generalised} shows that such a group $G$ must be almost simple, i.e., its socle $\mathrm{Soc}(G)$ is a non-abelian simple group. This reduction allows the classification problem to be approached by analyzing the possible simple groups that can occur as the socle. Subsequent work has ruled out many possible socles: Bamberg and Evans \cite{bamberg-evans2021no} eliminated all sporadic groups; Feng and Lu \cite{feng2023finitePSL} dealt with the case of $\PSL_2(q)$ (with the only case $q=9$, yielding $W(3,2)$); Lu, Zhang and Zou \cite{lu2024nonexistencePSU} excluded $\PSU_3(q)$ for $q\ge 3$; and Arumugam, Bamberg and Giudici \cite{ArumugamBambergGiudici2025LowRankGQs} excluded Suzuki and small Ree groups.

In this paper, we address the  case where the socle is the projective symplectic group $\mathrm{PSp}_4(q)$ with $q \ge 3$. This case is of particular interest because $\mathrm{PSp}_4(q)$ is naturally associated with the symplectic quadrangle $W(3,q)$. A key challenge in this case, which also distinguishes it from earlier work such as \cite{feng2023finitePSL, lu2024nonexistencePSU}, is that an automorphism group may fail to act transitively on the substructures fixed by an element---a situation that previous methods were not designed to handle. To overcome this, we develop new geometric tools (Lemmas \ref{lem:P-two-orbit-quadra} and \ref{lem:two-orbits-P-and-L-subqua}) that allow us to analyze the fixed substructure of an automorphism even in the presence of two orbits, thereby completing the classification. Our main result is as follows:

\begin{theorem}\label{main}
Let $\mathcal{S}$ be a finite thick generalized quadrangle and let $G \le \operatorname{Aut}(\mathcal{S})$ act primitively on both points and lines. If $G$ is almost simple with socle $\operatorname{PSp}_4(q)$, where $q \ge 3$, then, up to duality, either $\mc{S}\cong W(3,q)$ for $q\geq 3$, or $q=3$ and $\mc{S}\cong H(3,4)$.
\end{theorem}

This paper is organized as follows. In Section \ref{sec_prelim}, we present some preliminary results on generalized quadrangles and permutation groups. In Section \ref{Subsection_maximal_subgroups}, we collect information about the maximal subgroups and conjugacy classes of $\PSp_4(q)$. Finally, in Section \ref{sec_pfmain}, we give the proof of the main theorem.

\section{Background on generalized quadrangles and permutation groups}\label{sec_prelim}

In this section we recall some useful facts about both finite generalized quadrangles and permutation groups that will be needed later. Throughout this paper, we use $\mc{S}=(\mathcal{P}, \mathcal{L}, \mathcal I)$ to denote a finite generalized quadrangle of order $(s,t)$  with point set $\mc{P}$, line set $\mc{L}$, and incidence relation $\mathcal I$.  If we interchange the roles of points and lines in $\mc{S}$, we obtain another generalized quadrangle called the {\it dual} of $\mc{S}$, which has order $(t,s)$. When $s=t$, we simply say that $\mathcal{S}$ has order $s$.

\subsection{Parameters of generalized quadrangles}\label{GQ}

We briefly recall the standard parameter notation for finite generalized quadrangles and record several elementary facts that will be used later.

\begin{lemma}[{\cite[1.2.1, 1.2.2, 1.2.3]{PayneThas2009}}]\label{parameters}
Let $\mathcal{S}$ be a generalized quadrangle of order $(s,t)$ with $|\PA|$ points and $|\LA|$ lines.
The following hold:
\begin{enumerate}[label={(\arabic*)}]
\item $|\PA|=(s+1)(st+1)$ and $|\LA|=(t+1)(st+1)$; \label{vb}
\item $s+t$ divides $st(s+1)(t+1)$; \label{divcond}
\item if $s>1$ and $t>1$, then $s\leq t^2$ and $t\leq s^2$. \label{st2}
\end{enumerate}
\end{lemma}

\begin{lemma}[{\cite[Lemma 2.2 and Lemma 4.3]{lu2024nonexistencePSU}}]\label{2|P|^5>|L|^4_2222}
    Let $\mathcal{S}$ be a thick generalized quadrangle of order $(s,t)$ with $|\PA|$ points and $|\LA|$ lines. The following hold:
    \begin{enumerate}[label={(\arabic*)}]
        \item $|\PA|>(s+1)^2(t+1)/2$ and $|\LA|>(t+1)^2(s+1)/2$;
        \item $2^{-1/5}|\PA|^{4/5}<|\LA|<2^{1/4}|\PA|^{5/4}$.
    \end{enumerate}
\end{lemma}
For a positive integer $n$ and a prime $p$, let $n_p$ denote the $p$-part of $n$, i.e., the highest power of $p$ dividing $n$, and let $n_{p'} := n / n_p$ denote the $p'$-part of $n$.

\subsection{Permutation group theory}
This subsection recalls some standard notation and basic results from permutation group theory that will be used repeatedly. We assume the reader is familiar with fundamental notions such as transitive and primitive actions.

Let $G$ be a group acting on a set $\Omega$. For $\a\in\Omega$ and $g\in G$, we write $\a^g$ for the image of $\a$ under the permutation induced by $g$. We denote by $G_\a$ the stabilizer of $\a$ in $G$, and by $\a^G$ the $G$-orbit containing $\a$. The centralizer of $g$ in $G$ is written as $C_G(g)$, and the conjugacy class of $g$ in $G$ as $g^G$. Further terminology and notation can be found in \cite{DMpermutation}.
Given $g\in G$, define the fixed-point set $\Fix(g)=\{x\in\Omega: x^g=x\}$. The following lemma provides a convenient formula for computing $|\Fix(g)|$ in terms of conjugacy classes and point stabilizers, and it will be applied frequently throughout the paper.
\begin{lemma}[{\cite[Lemma 2.5]{liebeck1991minimal}}]\label{Fix_points_number}
Let $G$ be a finite transitive group acting on a set $\Omega$. Let $\a\in\Omega$ and $g\in G$. Then
\[|\Fix(g)|=\frac{|\Omega|\cdot|g^G\cap G_{\a}|}{|g^G|}.\]
\end{lemma}

The following lemma provides a criterion for  the centralizer to act transitively on the fixed-point set.

\begin{lemma}[{\cite[Lemma 2.4]{lu2024nonexistencePSU}}]\label{CT(g) acts transitively}
Let $G$ be a group acting transitively on a set $\Omega$ and let $g$ be a nontrivial element of $G$ fixing some $\a\in\Omega$. Then $C_G(g)$ acts transitively on $\Fix(g)$ if and only if $g^G\cap G_{\a}=g^{G_{\a}}$. In this case, $|\Fix(g)|=\dfrac{|C_G(g)|}{|C_{G_\a}(g)|}$.
\end{lemma}


\subsection{Substructures of a generalized quadrangle fixed by an automorphism}

Two points $x_1, x_2 \in \mathcal{P}$  are said to be collinear, denoted $x_1 \sim x_2$, if there exists a line $\ell \in \mathcal{L}$ incident with both. Dually, two lines $\ell_1, \ell_2 \in \mathcal{L}$ are concurrent, denoted $\ell_1 \sim \ell_2$, if there exists a point $x \in \mathcal{P}$ incident with both. We write $x \in \ell$ to indicate that point $x$ is incident with line $\ell$.

A generalized quadrangle $\mc{S'}=(\mc{P}',\mc{L}',\mathcal{I}')$ is called a \emph{subquadrangle} of $\mc{S}=(\mc{P},\mc{L},\mathcal{I})$ if $\mc{P}'\subseteq \mc{P}$, $\mc{L}'\subseteq\mc{L}$ and  $\mathcal{I}'$ is the restriction of $\mathcal{I}$ on $(\mc{P}'\times \mc{L}')\cup (\mc{L}'\times \mc{P}')$.
A \emph{grid} with parameters $(s_{1},s_{2})$ is a point-line incidence structure $(\mathcal{X}, \mathcal{B}, {\rm I})$ with
$$\mathcal{X}=\left\{x_{ij}: 0 \leqslant i \leqslant s_{1}, 0 \leqslant j \leqslant s_{2}\right\}, \qquad  \mathcal{B}=\left\{\ell_{0}, \ldots, \ell_{s_{1}}, \ell_{0}^{\prime}, \ldots, \ell_{s_{2}}^{\prime}\right\}$$
such that $x_{i j}{\rm I} \ell_{k}$ if and only if $i=k$, and $x_{ij} {\rm I} \ell_{k}^{\prime}$ if and only if $j=k$. A \emph{dual grid} with parameters $(s_1,s_2)$ is the point-line dual of a grid with parameters $(s_2,s_1)$. Note that in a grid, each point lies on exactly two lines.


\begin{lemma}[{\cite[2.4.1]{PayneThas2009}}]\label{subquadrangle_structure}
Let $g$ be an automorphism of a finite generalized quadrangle $\mathcal{S}=(\mathcal{P}, \mathcal{L})$ of order $(s,t)$. Let $\mathcal{P}_{g}$ and $\mathcal{L}_g$ be the set of fixed points and fixed lines of $g$ respectively, and let $\mathcal{S}_{g}=(\cP_g,\cL_g)$ be the induced incidence substructure on $\mathcal{P}_{g}\times\mathcal{L}_g$. Then  one of the following holds:
\begin{enumerate}
\item[(0)] $\mathcal{P}_{g}=\mathcal{L}_{g}=\varnothing$,
\item[(1)] $\mathcal{L}_{g}=\varnothing$, $\mathcal{P}_{g}$ is a nonempty set of pairwise noncollinear points,
\item[(1')] $\mathcal{P}_{g}=\varnothing$, $\mathcal{L}_{g}$ is a nonempty set of pairwise nonconcurrent lines,
\item[(2)] $\mathcal{P}_{g}$ contains a point $x$ such that $x \sim x'$ for every point $x' \in \mathcal{P}_{g}$, $\mathcal{L}_{g}$ is nonempty and each line of $\mathcal{L}_{g}$ is incident with $x$,
\item[(2')] $\mathcal{L}_{g}$ contains a line $\ell$ such that $\ell \sim \ell'$ for every line $\ell' \in \mathcal{L}_{g}$, $\mathcal{P}_{g}$ is nonempty and each point of $\mathcal{P}_{g}$ is incident with $\ell$,
\item[(3)] $\mathcal{S}_{g}$ is a grid with parameters $(s_1,s_2)$, $s_{1}<s_{2}$,
\item[(3')] $\mathcal{S}_{g}$ is a dual grid with parameters $(s_1,s_2)$, $s_{1}<s_{2}$,
\item[(4)] $\mathcal{S}_{g}$ is a subquadrangle of order $\left(s^{\prime}, t^{\prime}\right)$.
\end{enumerate}
\end{lemma}

\begin{lemma}[{\cite[Corollary 2.3, Lemma 2.5]{feng2023finitePSL}}]\label{is a subquadrangle}
    Let $g$ be an automorphism of a thick generalized quadrangle $\mathcal{S}$, and let  $\mc{S}_g=(\PA_g,\LA_g)$ be the fixed substructure of $g$. If $|\PA_g|\geq2$, $|\LA_g| \geq 2$ and $\mc{S}_g$ admits an automorphism group $H$ that is transitive on both its points and its lines, then 
    \begin{enumerate}
        \item[(1)]  $\mc{S}_g$ is a generalized quadrangle of order $(s',t')$ for some positive integers $s',t'$;
        \item[(2)] $H$ is nonabelian. 
    \end{enumerate}
\end{lemma}

\begin{lemma}\label{lem:P-two-orbit-quadra}
 Let $g$ be an automorphism of a thick generalized quadrangle $\mathcal{S}$, and let  $\mc{S}_g=(\PA_g,\LA_g)$ be the fixed substructure of $g$. Assume that $|\mathcal{P}_g|\ge 2$ and $|\mathcal{L}_g|\ge 2$. If there exists a subgroup $H\le \operatorname{Aut}(\mathcal{S}_g)$ such that $H$ is transitive on $\mathcal{L}_g$ and has exactly two orbits on $\mathcal{P}_g$, both of equal size at least $2$, then $\mathcal{S}_g$ is a subquadrangle.
\end{lemma}
\begin{proof}
By Lemma \ref{subquadrangle_structure}, $\mc{S}_{g}$ must be one of the types  (0)-(4). We show that all types except (4) lead to contradictions under the given hypotheses. Types (0),(1) and (1') are impossible since $|\mathcal{P}_g|,|\mathcal{L}_g|\ge 2$. If $\mathcal{S}_g$ is of type (2), then all lines in $\mc{S}_g$ pass through a common point $x$. This forces $x$ to be fixed by $H$, so $\{x\}$ is an $H$-orbit on $\mathcal{P}_g$ of size $1$, contradicting the assumption that $H$ has two orbits each of size at least $2$. Dually, type (2') yields a fixed line $\ell$, giving a singleton $H$-orbit in $\mathcal{L}_g$, which contradicts the transitivity of $H$ on $\mathcal{L}_g$ together with $|\mathcal{L}_g|\ge 2$. If $\mathcal{S}_g$ is of type (3), it is a grid with two line families of sizes $s_1+1$ and $s_2+1$ ($s_1<s_2$). The different sizes would force the two families to lie in different orbits, which is impossible. Dually, if $\mathcal{S}_g$ is of type (3'), it is a dual grid with point families of sizes $s_1+1$ and $s_2+1$; the hypothesis that $H$ has two equal-sized point orbits then forces $s_1+1=s_2+1$, contradicting $s_1<s_2$.
Thus none of the types (0)–(3') can occur, leaving type (4) as the only possibility. Hence $\mathcal{S}_g$ is a subquadrangle.
\end{proof}

\begin{lemma}\label{lem:two-orbits-P-and-L-subqua}
Let $g$ be an automorphism of a thick generalized quadrangle $\mathcal{S}$, and let  $\mc{S}_g=(\PA_g,\LA_g)$ be the fixed substructure of $g$. Suppose there exists $H\le\operatorname{Aut}(\mathcal{S}_g)$ such that $H$ has exactly two orbits $\mathcal{P}_1,\mathcal{P}_2$ on $\mathcal{P}_g$ and exactly two orbits $\mathcal{L}_1,\mathcal{L}_2$ on $\mathcal{L}_g$, with $|\mathcal{P}_1|,|\mathcal{P}_2|,|\mathcal{L}_1|,|\mathcal{L}_2|\ge 2$. If there exist indices $i,j,k,\ell\in\{1,2\}$ such that $|\mathcal{L}_i| > |\mathcal{P}_j|$ and $|\mathcal{P}_k| > |\mathcal{L}_\ell|$, then $\mathcal{S}_g$ is a subquadrangle.
\end{lemma}

\begin{proof}
By Lemma~\ref{subquadrangle_structure}, $\mathcal{S}_g$ must be one of the types (0)–(4). Types (0), (1), (1'), (2) and (2') are impossible by the same reasoning as in Lemma~\ref{lem:P-two-orbit-quadra}.
If $\mathcal{S}_g$ is a grid (type (3)), then the two line orbits are the two line families. In a grid, every point lies on exactly one line from each family. By $H$-transitivity on each line orbit, each line in $\mathcal{L}_j$ contains the same positive number of points from $\mathcal{P}_i$, so $|\mathcal{P}_i|$ is a multiple of $|\mathcal{L}_j|$ for all $i,j$. Hence $|\mathcal{P}_i| \ge |\mathcal{L}_j|$  holds for every pair $(i,j)$, contradicting the existence of $|\mathcal{L}_i| > |\mathcal{P}_j|$.
If $\mathcal{S}_g$ is a dual grid (type (3')), then by duality the same argument shows that $|\mathcal{L}_j| \ge |\mathcal{P}_i|$ for all $i,j$, contradicting the hypothesis that some point orbit exceeds some line orbit.
Thus $\mathcal{S}_g$ must be of type (4), i.e., a subquadrangle.
\end{proof}

\begin{lemma}[Benson's Theorem {\cite[1.9.1, 1.9.2]{PayneThas2009}}]\label{lem:benson}
Let $g$ be an automorphism of a finite generalized quadrangle $\mathcal{S}$ of order $(s,t)$, and let  $\mc{S}_g=(\PA_g,\LA_g)$ be the fixed substructure of $g$.  Define
\(\mathcal{L}_1(g)=\{\,\ell\in \mathcal{L}\mid \ell^g\neq \ell\sim \ell^g\,\}\).
Then
\[
(s+1)|\mathcal{L}_g|+|\mathcal{L}_1(g)|\equiv st+1 \pmod{s+t}.
\]
\end{lemma}

\begin{lemma}\label{lem:no_2_3_elemen}
      Suppose that $\mathcal{S}$ is a thick generalized quadrangle of order $s$ and that $G$ is an automorphism group of $\mathcal{S}$ acting transitively on both points and lines. Then there is no element $g \in G$ of order two or three such that $g^G \cap G_\a =g^G \cap G_\ell=\varnothing$ for some point $\a$ and line $\ell$ in $\mc{S}$.
\end{lemma}
\begin{proof}
Suppose such an element $g$ exists. Since $G$ acts transitively on points, applying Lemma~\ref{Fix_points_number} to the point set $\mathcal{P}$ gives $|\mathcal{P}_g| = |\mathcal{P}| \cdot |g^G \cap G_\alpha| / |g^G| = 0$, hence $\mathcal{P}_g = \varnothing$. Similarly, transitivity on lines yields $\mathcal{L}_g = \varnothing$. Define \(\mathcal{L}_1(g)=\{\,\ell\in \mathcal{L}\mid \ell^g\neq \ell\sim \ell^g\,\}\). We claim that \(\mathcal{L}_1(g)=\varnothing\). Indeed, let \(\ell\in\mathcal{L}_1(g)\). If \(g\) is an involution, then the point \(\ell\cap\ell^g\) is fixed by \(g\). If \(g\) has order three, then \(\ell,\ell^g,\ell^{g^2}\) are pairwise concurrent, and since a generalized quadrangle contains no triangles, they must all pass through a common point \(P\), which is fixed by \(g\). In either case, this contradicts \(\mathcal{P}_g=\varnothing\). Now apply Lemma \ref{lem:benson}. Since $|\mathcal{L}_g| = 0$ and $|\mathcal{L}_1(g)| = 0$, we have
    \[
    0 = (s+1)|\mathcal{L}_g| + |\mathcal{L}_1(g)| \equiv s^2+1 \pmod{2s},
    \]
    which is impossible for $s>1$. Hence no such $g$ exists.
\end{proof}

\section{Maximal subgroups and conjugacy classes of $\PSp_4(q)$}\label{Subsection_maximal_subgroups}
In this section we collect some information on $\PSp_4(q)$ that will be used throughout the paper. The projective symplectic group $\PSp_4(q)$ is defined as the quotient $\Sp_4(q)/Z(\Sp_4(q))$, where $Z(\Sp_4(q))$  denotes the center of $\Sp_4(q)$. 
Let $\hat{~}$ denote the natural projection from $\Sp_4(q)$ onto $\PSp_4(q)$. For a subgroup $H \leq \Sp_4(q)$, we write $\hat{}H$ for its image under this projection; for an element   $g \in\Sp_4(q)$, we write  $\hat{g}$ for its image. We use the standard ATLAS notation for group structures and extensions \cite{ConwayAtlas}, and follow \cite{bray2013maximal} for the descriptions of the maximal subgroups of $\PSp_4(q)$.

\subsection{Maximal subgroups of almost simple groups $G$ with socle $\PSp_4(q)$}
Let $G$ be an almost simple group with socle $\PSp_4(q)$. Then the complete list of maximal subgroups of $G$ is known.

\begin{lemma}\cite[Table 8.12, Table 8.13 and Table 8.14]{bray2013maximal}\label{maximal_subgroups}
    Let $G$ be an almost simple group with socle $T=\PSp_4(q)$, where $q=p^f\geq3$ for some prime  $p$ and integer $f$. Let $M$ be a maximal subgroup of $G$ not containing $T$. Then $M \cap T$ is isomorphic to one of the groups listed in Table \ref{tab:maximal_subgroups}.
\end{lemma}

\begin{table}[!htbp]
\aboverulesep=0pt \belowrulesep=0pt
\setlength{\abovecaptionskip}{0cm}
\setlength{\belowcaptionskip}{0cm}
\caption{The possibilities for $M\cap T$}
\label{tab:maximal_subgroups}
\centering
\begin{adjustbox}{center,max width=\textwidth}
$ \small
\setlength{\arraycolsep}{5pt}
\begin{array}{llll}
\toprule
\text{Type} & M \cap T & \text{Condition} & [T:M \cap T] \\
\midrule
&&&\\[-0.4cm]
\text{(a)} & \hat~ [q]^{1+2}:((q-1)\times \mathrm{Sp}_2(q)) & q\ \text{odd} & (q^2+1)(q+1)\\[0.1cm]
\text{(b)} & \hat~ [q]^3:\mathrm{GL}_2(q) & & (q^2+1)(q+1)\\[0.1cm]
\text{(c)} &  [q^4]:\mathrm{C}_{q-1}^2 & q\ \text{even} & (q^2+1)(q+1)^2\\[0.1cm]
\text{(d)} & \hat~ \mathrm{GL}_2(q).2 & q\geq5,\ q\ \text{odd} & q^3(q^2+1)(q+1)/2\\[0.1cm]
\text{(e)} & \hat~ \mathrm{Sp}_2(q)^2:\mathrm{S}_2 & & q^2(q^2+1)/2\\[0.1cm]
\text{(f)} & {\mathrm{C}_{q+\epsilon}}^2:\mathrm{D}_8,\ \epsilon=\pm1 & q\ \text{even} & q^4(q-\epsilon)^2(q^2+1)/8\\[0.1cm]
\text{(g)} & \hat~ \mathrm{Sp}_2(q^2):2 & & q^2(q^2-1)/2\\[0.1cm]
\text{(h)} & \hat~ \mathrm{GU}_2(q).2 & q\geq5,\ q\ \text{odd} & q^3(q^2+1)(q-1)/2\\[0.1cm]
\text{(i)} &  \mathrm{C}_{q^2+1}:4 & q\ \text{even} & q^4(q^2-1)^2/4\\[0.1cm]
\text{(j)} & \hat~ \mathrm{Sp}_4(q_0).(2,r) & q\ \text{odd},\ q=q_0^r,\ r\ \text{prime} &
\dfrac{q^{4}(q^4-1)(q^2-1)}{q_0^4(q_0^4-1)(q_0^2-1)(2,r)}\\[0.4cm]
\text{(k)} & \mathrm{Sp}_4(q_0) & q\ \text{even},\ q=q_0^r,\ r\ \text{prime} &
\dfrac{q^{4}(q^4-1)(q^2-1)}{q_0^4(q_0^4-1)(q_0^2-1)}\\[0.3cm]
\text{(l)} & \hat~ 2_-^{1+4}.\mathrm{S}_5 & q=p\equiv \pm1\pmod 8  &
q^4(q^4-1)(q^2-1)/3840\\[0.1cm]
\text{(m)} & \hat~ 2_-^{1+4}.\mathrm{A}_5 & q=p\equiv \pm3\pmod 8  &
q^4(q^4-1)(q^2-1)/1920\\[0.1cm]
\text{(n)} & \mathrm{SO}_4^{\epsilon}(q),\ \epsilon=\pm1 &  q\ \text{even} & q^2(q^2+\epsilon)/2\\[0.1cm]
\text{(o)} & \mathrm{A}_6 &  q=p\equiv 5,7\pmod{12},\ q\neq 7&
q^4(q^4-1)(q^2-1)/720\\[0.1cm]
\text{(p)} & \mathrm{S}_6 &  q=p\equiv 1,11\pmod{12} &
q^4(q^4-1)(q^2-1)/1440\\[0.1cm]
\text{(q)} & \mathrm{A}_7 & q=7 & 54880\\[0.1cm]
\text{(r)} &  \mathrm{PSL}_2(q) & p\geq5,\ q\ \text{odd} & q^3(q^4-1)\\[0.1cm]
\text{(s)} & \mathrm{Sz}(q) & q\ \text{even},\ f\geq3\ \text{odd} & q^2(q+1)^2(q-1)\\

\bottomrule
\end{array}
$
\end{adjustbox}
\end{table}

\subsection{Conjugacy classes of elements in $\PSp_4(q)$}
In this subsection, we list some conjugacy classes of elements in $T=\PSp_4(q)$.

\begin{lemma}\label{Conj_Order=2}
Let \(T=\PSp_4(q)\) with $q$ odd. Let \(H\) be a subgroup of \(\Sp_4(q)\) containing \(Z(\Sp_4(q))\), and let \hspace{1pt} \(\hat{}H\) be its image in \(T\). The following hold: 
\begin{enumerate}
    \item[(1)] \(T\) has exactly two conjugacy classes of involutions, denoted \(Y_1\) and \(Y_2\). The class \(Y_1\) consists of the images of involutions in \(\Sp_4(q)\) under the natural projection, while \(Y_2\)   consists of the images of elements of order \(4\) in \(\Sp_4(q)\) whose squares are the nontrivial element $-\mathrm{I}_4$ of \(Z(\Sp_4(q))\). Moreover,
\[
|Y_1|=\frac{q^2}{2}(q^2+1). 
\]
\item[(2)] If $H$ contains $m$ involutions, then $|Y_1 \cap \hat{}H|=(m-1)/2$.
\end{enumerate}
\end{lemma}

\begin{proof}
Part (1) follows from  \cite[Lemma 2.4]{Involution_wong1969characterization} and \cite[Section 3-4]{Involution_everett2020commuting}. For (2), by (1), the set \(Y_1\cap \hat{}H\) consists precisely of the images under  the natural projection of the involutions in \(H\). Since \(Z(\Sp_4(q))=\{I_4,-I_4\}\) and \(\hat{-I_4}\) equals the identity element of \(T\), the central involution \(-I_4\in H\) does not contribute to \(Y_1\cap \hat{}H\). Now let \(x\in H\) be an involution with \(x\neq -I_4\). Then \(-x=x(-I_4)\in H\) is also an involution, \(x\neq -x\), and \(\hat{x}=\hat{-x}\). Thus the \(m-1\) involutions in \(H\setminus\{-I_4\}\) are partitioned into pairs \(\{x,-x\}\), and each pair yields a single element of \(Y_1\cap \hat{}H\). Therefore \(|Y_1\cap \hat{}H|=(m-1)/2\), as required.
\end{proof}

\begin{lemma}\cite[Proposition 6.3]{liebeck1996classical}\label{Conj_order=3}
    Let \(T=\PSp_4(q)\), where $q=p^f$ with $p\geq5$ prime. Then $T$ has exactly two conjugacy classes of elements of order 3. 
\end{lemma}

\begin{lemma}\label{Conj_Order=p}
Let \(T=\PSp_4(q)\), where \(q=p^f\) with \(p\ge 5\) prime. Then \(T\) has exactly six conjugacy classes of elements of order \(p\). Moreover, there exists a conjugacy class \(g^T\) such that \(C_T(g)\) is abelian and \(|C_T(g)|=q^2\).
\end{lemma}

\begin{proof}
    It follows from \cite[Proposition 3.4.10]{burness2016classical_derangement} that there is a bijection from the set of $T$-classes of elements of order $p$ in $T$ and the set of  nontrivial signed partition of $4$ of the form $(p^{a_p},\epsilon_{p-1}(p-1)^{a_{p-1}},(p-2)^{a_{p-2}}, \cdots, \epsilon_2 2^{a_2},1^{a_1})$ where $\sum_iia_i=4$, $\epsilon_{2i}=\pm$ and $a_i$ is even when $i$ is odd. For \(p\ge 5\), the only possible nontrivial partitions of $4$ satisfying these conditions are 
    \((\pm2,1^2)\), \((\pm2^2)\), and \((\pm4)\). Hence there are exactly six such signed partitions, and consequently six conjugacy classes of elements of order \(p\) in \(T\). Consider the type \((\pm4)\), which corresponds to the Jordan block \(J_4\) in \(\Sp_4(q)\). Let \(\hat{g}=J_4\) and set \(g=\hat{g}Z(\Sp_4(q))\in T\).  By \cite[Lemma 3.4.11]{burness2016classical_derangement}, we have \(|C_{\Sp_4(q)}(\hat g)|=2q^2\).  Since \(Z(\Sp_4(q))\) is contained in \(C_{\Sp_4(q)}(\hat g)\), it follows that \(|C_T(g)| = |C_{\Sp_4(q)}(\hat g)|/|Z(\Sp_4(q))| = q^2\). One readily checks that \(C_{\GL_4(q)}(\hat g)\) is abelian. Therefore, its subgroup \(C_{\Sp_4(q)}(\hat g)\) is also abelian, and consequently \(C_T(g)\) is abelian.
\end{proof}

\begin{lemma}\label{Order_4_p=2}
 Let \(T=\PSp_4(q)\), where \(q=2^f\). Then \(T\) has exactly two conjugacy classes of elements of order $4$. Let \(g_1\) and \(g_2\) be representatives of these two classes. Then, for each \(i\in\{1,2\}\), the centralizer \(C_T(g_i)\) is abelian and $|C_T(g_i)|=2q^2$.
\end{lemma}

\begin{proof}
 From \cite[Table~1]{p=2_Order_4_alavi2022quantitative}, the group \(T\) has exactly two conjugacy classes of elements of order \(4\), denoted \(A_{41}\) and \(A_{42}\). Both classes have size \(q^2(q^4-1)(q^2-1)/2\). 
For even \(q\), we have \(\operatorname{PSp}_4(q) \cong \operatorname{Sp}_4(q)\), and the latter can be defined as
\[
\operatorname{Sp}_4(q) = \{ A \in \operatorname{GL}_4(q) \mid A W A^\top = W \},
\]
where
\[
W = \begin{pmatrix}
0 & 0 & 0 & 1 \\
0 & 0 & 1 & 0 \\
0 & 1 & 0 & 0 \\
1 & 0 & 0 & 0
\end{pmatrix}.
\]

Following \cite[Table~1]{p=2_Order_4_alavi2022quantitative}, we take the representatives
\[
g_1 = \begin{pmatrix}
1 & 1 & 1 & 0 \\
0 & 1 & 1 & 0 \\
0 & 0 & 1 & 1 \\
0 & 0 & 0 & 1
\end{pmatrix} \in A_{41},
\qquad
g_2 = \begin{pmatrix}
1 & 1 & 1 & \zeta \\
0 & 1 & 1 & 0 \\
0 & 0 & 1 & 1 \\
0 & 0 & 0 & 1
\end{pmatrix} \in A_{42},
\]
where \(\zeta \in \mathbb{F}_q\) is a fixed element. A direct computation shows that
\[
C_{\operatorname{GL}_4(q)}(g_1) = C_{\operatorname{GL}_4(q)}(g_2) =
\left\{
\begin{pmatrix}
a & b & b+c & d \\
0 & a & b & c \\
0 & 0 & a & b \\
0 & 0 & 0 & a
\end{pmatrix}
\ \middle|\ a \in \mathbb{F}_q^*,b,c,d \in \mathbb{F}_q
\right\},
\]
which is an abelian group. Consequently, \(C_{\operatorname{Sp}_4(q)}(g_i) = C_{\operatorname{GL}_4(q)}(g_i) \cap \operatorname{Sp}_4(q)\) is also abelian.  Moreover, $|C_T(g_i)|=\frac{|T|}{|g_i^T|}=2q^2$.
\end{proof}

\section{Proof of Theorem \ref{main}}\label{sec_pfmain}
This section is devoted to proving Theorem \ref{main}.
Let $\mathcal{S}=(\mathcal{P}, \mathcal{L}, \mathcal{I} )$ be a finite thick generalized quadrangle of order $(s,t)$  admitting an automorphism group $G$ that acts primitively on both points and lines. Assume that $G$ is  almost simple  with socle $T=\PSp_4(q)$, where $q=p^f\geq 3$ for some prime $p$ and integer $f$. For $\a\in \mathcal{P}$ and $\ell\in \mathcal{L}$, set
\[T_{\a}=G_{\a}\cap T \text{  and  }  T_\ell=G_\ell\cap T.\]
Since $G$ is primitive and $T\trianglelefteq G$, the group $T$ acts transitively on $\mathcal{P}$ and $\mathcal{L}$. Consequently,
\begin{align}
|\mathcal{P}|=&[T:T_{\a}]= (s+1)(st+1),\label{number_point}\\
|\mathcal{L}|=&[T:T_\ell]= (t+1)(st+1).\label{number_line}
\end{align}

By the maximality of $G_\alpha$ and $G_\ell$ in $G$, the subgroups $T_\alpha$ and $T_\ell$ appear in the list of subgroups of $T$ given in Table~\ref{tab:maximal_subgroups}. For convenience, we say that $T_\alpha$ (or $T_\ell$) is of type (x) if it is isomorphic to the group in row (x) of the table. For instance, $T_\alpha$ has type (a) if $T_\alpha\cong \hat~ [q]^{1+2}:((q-1)\times \mathrm{Sp}_2(q))$. For an integer $n$ and a prime $p$, we denote by $n_p$ the highest power of $p$ dividing $n$.

In the next two subsections, we consider two separate cases depending on whether $T_\alpha$ and $T_\ell$ are of the same type or not.

\subsection{The case when $T_\a$ and $T_\ell$ are of the same type}\label{section:same-type}
In this subsection, we assume that $T_\alpha$ and $T_\ell$ are of the same type as listed in Table~\ref{tab:maximal_subgroups}, and we treat each type case by case. By (\ref{number_point}) and (\ref{number_line}), the values of $|\PA|$ and $|\LA|$ are given in the last column of Table \ref{tab:maximal_subgroups}. In some cases, the equations $|\PA|=(s+1)(st+1)$ and $|\LA|=(t+1)(st+1)$ yield immediate arithmetic contradictions. When this fails, we analyze the fixed structures $\mc{S}_g=(\PA_g, \LA_g)$ of a certain special element $g \in T$. Combined with group-theoretic arguments or the determination of possible subquadrangles, this leads to a contradiction.

\begin{lemma}\label{type_a}
    The subgroup $T_\a$ cannot be of type (a).
\end{lemma}
\begin{proof}
   Suppose, for a contradiction, that $T_\alpha$ is of type (a) with $q$ odd.  From Table \ref{tab:maximal_subgroups}, $|\PA|=|\mathcal{L}|=(q^2+1)(q+1)=(s+1)(s^2+1)$, which  forces $s = q$. According to \cite[Table 8.12 and Table 8.13]{bray2013maximal},  there is a unique conjugacy class of subgroups isomorphic to $\hat~ [q]^{1+2}:((q-1)\times \mathrm{Sp}_2(q))$ in $T$. Hence, after replacing \(T_\alpha\) by a suitable conjugate, we may assume \(T_\alpha = T_\ell\). Set $H := T_\alpha = T_\ell$. Geometrically, $H$ is the stabilizer of a $1$-dimensional totally isotropic subspace in the $4$-dimensional symplectic space 
    (see \cite[Table 2.3]{bray2013maximal}). Now, $H$ fixes $\alpha$ and acts on the line set $\mathcal{L}$. By \cite[Theorem 8.2]{taylor1992geometry}, \(H\) has exactly three orbits on \(\mathcal{L}\), of sizes \(1\), \(q(q+1)\), and \(q^3\).  Let $\Lambda$ be the set of $q+1$ lines through $\a$. Since $H$ fixes $\alpha$, it leaves $\Lambda$ invariant. Thus $\Lambda$ is a union of $H$-orbits on $\mathcal{L}$, which is impossible since $|\Lambda| = q+1$ cannot be expressed as a sum of any subset of $\{1, q(q+1), q^3\}$. Therefore, $T_\a$ cannot be of type (a).
\end{proof}

\begin{lemma}\label{W(3,q)-q-even}
    If the subgroup $T_\alpha$ is of type \emph{(b)}, then $\mc{S}$ is the classical generalized quadrangle $W(3,q)$ for even $q\ge 4$.
\end{lemma}
\begin{proof}
   Suppose that $T_\alpha$ is of type (b). From Table \ref{tab:maximal_subgroups}, we have $|\PA|=|\mathcal{L}|=(q^2+1)(q+1)=(s+1)(s^2+1)$,  hence $s=q$. We first consider the case where $q$ is odd. According to \cite[Tables 8.12 and 8.13]{bray2013maximal}, there is a unique conjugacy class of subgroups isomorphic to $\hat{}q^3:\mathrm{GL}_2(q)$. Thus, after replacing \(T_\alpha\) by a suitable conjugate, we may assume \(T_\alpha = T_\ell\). The following argument is similar to that in the proof of Lemma \ref{type_a}. According to the geometric description in \cite[Table 2.3]{bray2013maximal}, $T_\alpha$ is the stabilizer of a $2$-dimensional totally isotropic subspace in the $4$-dimensional symplectic space. Moreover, $T_\alpha=T_\ell$ has exactly three orbits on $\LA$, of sizes  \(1\), \(q(q+1)\), and \(q^3\), which yields the same contradiction as in type (a).

Therefore \(q\) must be even. By \cite[Table 8.14]{bray2013maximal}, there are exactly two conjugacy classes of subgroups isomorphic to $\hat{}q^3:\mathrm{GL}_2(q)$. If $T_\a$ and $T_\ell$ are conjugate in $T$, then the same argument as in the case where $q$ is odd yields a contradiction. Hence, by \cite[Notes in Table 8.14]{bray2013maximal} and up to duality, it follows that \(T_\alpha\) and \(T_\ell\) are the stabilizer of a $1$-dimensional totally isotropic subspace and a $2$-dimensional totally isotropic subspace in the $4$-dimensional symplectic space, respectively. By Witt's Lemma \cite[Theorem 7.4]{taylor1992geometry}, $T_\alpha$ has exactly two orbits on $\LA$, namely $\mathcal{O}_1=\{W\in \LA \mid \alpha \leq W\}$ and $\mathcal{O}_2=\{W \in \LA \mid \alpha \nleq W\}$. These orbits have sizes $q+1$ and $q^2(q+1)$, respectively. Since the set $\Lambda$ of $q+1$ lines through $\alpha$ is the union of $T_\alpha$-orbits, comparing the orbit sizes, we deduce that $\Lambda=\mathcal{O}_1$. Thus the incidence relation is given by inclusion, and consequently \(\mathcal{S}\) is isomorphic to \(W(3,q)\) for even \(q \ge 4\).  
\end{proof}

\begin{lemma}
    The subgroup $T_\a$ cannot be of type (c).
\end{lemma}
\begin{proof}
    Suppose, for a contradiction, that $T_\a$ is of type (c). From Table \ref{tab:maximal_subgroups},
    \begin{equation}\label{equa_type_c}
        |\PA|=|\mathcal{L}|=(q^2+1)(q+1)^2=(s+1)(s^2+1),
    \end{equation}
     where \(q \ge 4\) is even. Hence \(s\) must be even.  Reducing \eqref{equa_type_c}  modulo $q$ yields $s(s^2+s+1)\equiv0 \pmod q$. Hence $q $ divides $ s$ as $s^2+s+1$ is odd. Write $s=qm$ for some positive integer $m$. If $m=1$, i.e., $s=q$, then substituting into \eqref{equa_type_c} gives $q(q^2+1)(q+1)=0$, which is impossible. Thus $m\geq 2$. Reducing \eqref{equa_type_c} modulo $q^2$ further yields $q(m-2) \equiv 0 \pmod{q^2}$ This implies \(q \mid (m-2)\), so write \(m-2 = qk\) for some integer \(k \ge 0\). Then \(s = q^2k + 2q\).
If \(k = 0\), then \(s = 2q\). Substituting into \eqref{equa_type_c} yields
 \(q^2(q^2-6q-2) = 0\), which is impossible for \(q \ge 4\). Thus \(k \ge 1\).
 Then $s>q^2$ and so \[(s+1)(s^2+1)>s^3 > q^6>(q^2+1)(q+1)^2,\] contradicting \eqref{equa_type_c}. Therefore equation \eqref{equa_type_c}  admits no solution, and $T_\alpha$ cannot be of type (c).
\end{proof}

\begin{lemma}\label{type_d}
    The subgroup $T_\a$ cannot be of type (d).
\end{lemma}
\begin{proof}
    Suppose, for a contradiction, that $T_\a$ is of type (d). From Table \ref{tab:maximal_subgroups},
    \begin{equation*}
        |\PA|=|\mathcal{L}|=q^3(q^2+1)(q+1)/2=(s+1)(s^2+1),
    \end{equation*}
    where \(q\) is odd. Since $\gcd(s+1,s^2+1)=\gcd(s+1,2)$ divides $2$ and $q$ is odd, we have either $q^3 \mid (s+1)$ or $q^3 \mid (s^2+1)$. If $q^3 \mid (s+1)$, then $s+1\geq q^3$.
    Consequently, \[(s+1)(s^2+1) \ge q^3(q^6-2q^3+2)>q^3(q^2+1)(q+1)/2=|\PA|,\] a contradiction. Hence $q^3 \mid (s^2+1)$ and so $p \equiv 1\pmod{4}$. Then $q^2+1 \equiv 2\pmod{4}$ and $q+1\equiv 2\pmod{4}$, which implies that $(q^3(q^2+1)(q+1)/2)_2=2$. Therefore, $|\PA|$ is even and $s$ is odd. However, $((s+1)(s^2+1))_2\geq4>(q^3(q^2+1)(q+1)/2)_2$, which leads to a contradiction. Therefore  $T_\alpha$ cannot be of type (d).
\end{proof}

\begin{lemma}\label{type_e_f_g_i_n}
    The subgroup $T_\a$ cannot be of type (e), (f), (g), (i) or (n).
\end{lemma}
\begin{proof}
    Suppose, for a contradiction, that $T_\a$ is of type (e), (f), (g), (i) or (n). We first treat the cases of type (e) and (g) with $q$ odd.  The argument for type (g) is analogous to that for type (e), so we present only the details for type (e). Assume that \(T_\alpha\) is of type (e) with \(q\) odd. From Table \ref{tab:maximal_subgroups},
    \begin{equation}\label{|P|_type_e}
        |\PA|=|\mathcal{L}|=q^2(q^2+1)/2=(s+1)(s^2+1).
    \end{equation}
    Proceeding as in the proof of Lemma \ref{type_d}, we deduce that $q^2 \mid (s^2+1)$ and  write $s^2+1=q^2m$ for some integer $m\geq 1$. Substituting it into (\ref{|P|_type_e}) gives $(q^2+1)/2=(s+1)m$, which implies $s+1$ divides $q^2+1$, i.e. $q^2\equiv -1 \pmod {s+1}$. On the other hand, $-2 \equiv s^2+1 = q^2 m \equiv -m \pmod{s+1}$.  
Therefore, $m+2 \equiv 0 \pmod{s+1}$ and so $m\geq s-1$. From  \eqref{|P|_type_e} we have $(q^2+1)/2=(s+1)m\geq (s+1)(s-1) = s^2-1$, hence \(q^2 \ge 2s^2-3\). Using this lower bound for \(q^2\) in \eqref{|P|_type_e}  gives $(s+1)(s^2+1)=q^2(q^2+1)/2\geq (2s^2-3)(s^2-1)$.
    Expanding leads to $2s^4-s^3-6s^2-s+2\leq0$ and so $s=2$. Then \eqref{|P|_type_e} becomes $q^2(q^2+1)/2=15$, which has no integer solution for $q$. Therefore type (e) cannot occur for odd \(q\).
   
 
    
   Next we consider the cases of type (e), (g) and (n) simultaneously for even $q\geq4$. From Table \ref{tab:maximal_subgroups}, we have 
   \begin{equation}\label{type_e_even}
       |\PA|=|\mathcal{L}|=q^2(q^2+\epsilon)/2=(s+1)(s^2+1)
   \end{equation}
  where $\epsilon=1$ or $-1$. Since $q$ is even,  $s$ is odd, and hence $\gcd(s+1,s^2+1)=2$. It follows from   $(s^2+1)_2=2$ that $(s+1)_2=q^2/4$. A direct comparison shows that for $q\geq 8$,
  $$(s+1)(s^2+1)\geq \frac{q^2}{4}\left(\frac{q^4}{16}-\frac{q^2}{2}+2\right)> \frac{q^2(q^2+\epsilon)}{2}=|\PA|,$$
   a contradiction. For $q=4$, substituting it into (\ref{type_e_even}) gives $(s+1)(s^2+1)=136$ or $120$, neither of which has an integer solution $s$. Hence types (e), (g) and (n) are impossible for even $q$.

   Now we consider type (f) with $q$ even. In this case, $|\mathcal{P}| = |\mathcal{L}| = q^4(q-\epsilon)^2(q^2+1)/8=(s+1)(s^2+1)$ is even and so $s$ is odd. As in the above argument in the case of type (n), we deduce that $(s^2+1)_2=2$ and $(s+1)_2=q^4/16$. For $q \ge 8$, one checks that \[(s+1)(s^2+1) \geq \frac{q^4}{16}\left(\frac{q^8}{256}-\frac{q^4}{8}+2\right)> \frac{q^4(q-\epsilon)^2(q^2+1)}{8}=|\PA|,\] a contradiction.  For $q=4$, we have $(s+1)(s^2+1)=4896$ or $13600$, neither of which has an integer solution $s$. The argument for type (i) is analogous to that for type (f) with $q$ even, so we omit the details. This completes the proof.
\end{proof}

\begin{lemma}
    The subgroup $T_\a$ cannot be of type (h).
\end{lemma}
\begin{proof}
    Suppose, for a contradiction, that $T_\a$ is of type (h). From Table \ref{tab:maximal_subgroups}, we have 
    \begin{equation}\label{type_h_|P|}
       |\mathcal{P}| = |\mathcal{L}| = q^3(q^2+1)(q-1)/2 = (s+1)(s^2+1),
    \end{equation}
     where $q \geq 5$ is odd. If $s=2$, then (\ref{type_h_|P|}) has no solution for $q$. Hence we assume $s\geq 3$. Proceeding as in the proof of Lemma~\ref{type_d}, we deduce   $q^3 \mid (s^2+1)$   so $p\equiv 1\pmod{4}$. Write 
    \begin{equation}\label{type_h_s^2+1=q^3k}
        s^2+1=q^3k
    \end{equation}
    for some positive integer $k$. Substituting \eqref{type_h_s^2+1=q^3k} into \eqref{type_h_|P|} yields
    \begin{equation}\label{type_h_|P|_2}
        (q^2+1)(q-1)=2(s+1)k
    \end{equation}
    For \(q \ge 5\), one verifies that $\frac{3q^6}{4} < q^3(q^2+1)(q-1) < q^6$.  Moreover, for \(s \ge 3\), $2s^3 < 2(s+1)(s^2+1) < 3s^3$. Combining these with  \eqref{type_h_|P|} gives $q^6/4<s^3<q^6/2$. From (\ref{type_h_s^2+1=q^3k}), we deduce $(q^3k)^3/2<s^6 < (q^3k)^3$. 
    Hence $q^{12}/16< (q^3k)^3<q^{12}/2$, which gives rise to 
    \begin{equation}\label{k_bound}
     q/3<q/\sqrt[3]{16}<k<q/\sqrt[3]{2}<q.   
    \end{equation}
    Solving \eqref{type_h_|P|_2} for $s$ and substituting into \eqref{type_h_s^2+1=q^3k} leads to \[(q^2+1)^2(q-1)^2-4k(q^2+1)(q-1)+8k^2=4q^3k^3.\] Reducing it modulo $q^3$, we derive $8k^2+4k(q^2-q+1)+(3q^2-2q+1) \equiv 0 \pmod{q^3}$. Define $F(k):=8k^2+4k(q^2-q+1)+(3q^2-2q+1)$. The function \(F(k)\) is strictly increasing in \(k\) because its derivative \(F'(k) = 16k + 4(q^2 - q + 1) > 0\).  Hence, by \eqref{k_bound}, one checks that
    \begin{equation}\label{F(k)}
        q^3< F(q/3)<F(k)<F(q)< 6q^3
    \end{equation}
    Since $q\equiv 1\pmod{4}$, we deduce that $F(k)\equiv 2\pmod{4}$. Combining this with  $q^3 \mid F(k)$ and (\ref{F(k)}), we conclude that $F(k)=2q^3$, which can be reduced to 
    \begin{equation}\label{type_h_K}
        8k^2+4k(q^2-q+1)+(-2q^3+3q^2-2q+1)=0
    \end{equation}
    Consider (\ref{type_h_K}) as a quadratic in \(k\). Its discriminant is $\Delta=16(q^4+2q^3-3q^2+2q-1)$. If (\ref{type_h_K}) has an integer solution $k$, then $\Delta$ must be a perfect square. However, one readily checks that $16(q^2+q-2)^2<\Delta<16(q^2+q-1)^2$ so $\Delta$ cannot be a perfect square, a contradiction. This completes the proof.
\end{proof}

\begin{lemma}\label{typr_j}
    The subgroup $T_\a$ cannot be of type (j).
\end{lemma}
\begin{proof}
    Suppose that $T_\a$ and $T_\ell$ are both of type (j). Then $T_\a \cong \PSp_4(q_0).(2,r_0)$ and $T_\ell \cong \PSp_4(q_1).(2,r_1)$, where $q=q_0^{r_0}=q_1^{r_1}=p^f$ is odd and $r_0$, $r_1$ are primes. By duality, we may assume  $r_0\leq r_1$.
    
 \smallskip
    \textbf{Case 1: $2 < r_0 \leq r_1$ and $p \geq 5$.} In this case, $T_\a \cong \PSp_4(q_0)$ and $T_\ell \cong \PSp_4(q_1)$. By Lemma \ref{Conj_Order=p}, each of $T$, $T_\a$ and $T_\ell$ has exactly six conjugacy classes of elements of order $p$. Moreover, it follows from \cite[Proposition 4.1]{gon2017unipotent_class} that all the six conjugacy classes of $T$ are exactly represented in both $T_\a$ and $T_\ell$. Hence, for every $T$-conjugacy class $C$ of elements of order $p$, the intersections $C\cap T_\alpha$ and $C\cap T_\ell$ are single conjugacy classes in $T_\alpha$ and $T_\ell$, respectively. 
    Let $g \in T$ be an element of order $p$ such that $C_T(g)$ is abelian and $|C_T(g)|=q^2$ as in Lemma \ref{Conj_Order=p}. Since the class $g^T$ is represented in $T_\a$ and $T_\ell$, we can choose $g_1 \in g^T \cap T_\alpha$ and $g_2 \in g^T \cap T_\ell$.   Then there exists $c \in T$ such that $g_1^c=g_2$. Setting $\beta=\a^c$, we have $g_2 \in T_\beta\cap T_\ell$. Replacing \(\alpha\) by \(\beta\) if necessary, we may   assume  \(g\in T_{\alpha}\cap T_\ell\). Consequently, $g^T \cap T_\a= g^{T_\a}$ and $g^T \cap T_\ell=g^{T_\ell}$. It follows from Lemma~\ref{CT(g) acts transitively} that \(C_T(g)\) acts transitively on \((\mathcal{P}_{g},\mathcal{L}_{g})\). Hence, $|\PA_g|=\frac{|C_T(g)|}{|C_{T_\a}(g)|}=\frac{q^2}{q_0^2}>2$ and $|\LA_g|=\frac{|C_T(g)|}{|C_{T_\ell}(g)|}=\frac{q^2}{q_1^2}>2$. Lemma~\ref{is a subquadrangle} then implies that \(C_T(g)\) is nonabelian,  contradicting the choice of \(g\) with abelian centralizer.  Hence Case 1 cannot occur.
 
 \smallskip
    \textbf{Case 2: $2<r_0\leq r_1$ and $p=3$.} Then $T_\a \cong \PSp_4(q_0)$ and $T_\ell \cong \PSp_4(q_1)$. By Lemma \ref{Conj_Order=2}, each of $T$, $T_\a$ and $T_\ell$ has exactly two conjugacy classes of involutions. Let $g\in T$ be a representative of the conjugacy class $Y_1$ in Lemma \ref{Conj_Order=2}. Since $\Sp_4(q_0)$ and $\Sp_4(q_1)$ both contain $Z(\Sp_4(q))$, we deduce that $|g^T\cap T_\a|=|g^{T_\a}|$ and $|g^T \cap T_\ell|=|g^{T_\ell}|$ by Lemma \ref{Conj_Order=2}(2).  After conjugating in $T$ we may assume $g \in T_\alpha \cap T_\ell$, so that $g^T \cap T_\alpha = g^{T_\alpha}$ and $g^T \cap T_\ell = g^{T_\ell}$. Applying Lemma~\ref{Fix_points_number} we obtain \[
|\mathcal{P}_g| = \frac{|\mathcal{P}| \cdot |g^T \cap T_\alpha|}{|g^T|}
               = \frac{q^2 (q^2-1)^2}{q_0^2 (q_0^2-1)^2} > 2,
\qquad
|\mathcal{L}_g| = \frac{q^2 (q^2-1)^2}{q_1^2 (q_1^2-1)^2} > 2.
\] By Lemma \ref{is a subquadrangle}, $\mathcal{S}_g$ is a subquadrangle of order $(s',t')$. Set $A:=\frac{q(q^2-1)}{q_0(q_0^2-1)}$, $B:=\frac{q(q^2-1)}{q_1(q_1^2-1)}$. Then Lemma \ref{parameters}(1) yields \[
A^2 = (s'+1)(s't'+1),\qquad B^2 = (t'+1)(s't'+1).  
\]
Since $A$ and $B$ are odd,  $s'+1$, $t'+1$, $s't'+1$ are all odd.  Assume first that $r_0<r_1$. We deduce that $\gcd(s'+1,t'+1,s't'+1)=1$.  It follows that $s'+1$, $t'+1$, $s't'+1$ are perfect squares.
 Write $s'+1=a^2$, $t'+1=b^2$ and $s't'+1=c^2$. Then $(\frac{b}{a})_3=(\frac{B}{A})_3=\frac{q_0}{q_1}\geq 3$. So $3 \mid b$. Consequently $t'=b^2-1\equiv 2 \pmod  3$.   If $3 \mid a$, then $s'=a^2-1\equiv 2 \pmod 3$ and thus $c^2=s't'+1\equiv 2\pmod  3$, which is impossible. Therefore $3 \nmid a$, so $s'=a^2-1 \equiv 0\pmod   3$ and $c^2 \equiv 1 \pmod  3$. Then   $3 \nmid ac=A$ contradicting $(A)_3=\frac{q}{q_0}\geq 3$. Thus $r_0 < r_1$ is impossible. Now if $r_0=r_1$, then $s'=t'$ and $A^2=(s'+1)(s'^2+1)$. By the same argument as above, $s'^2+1$ must be a perfect square, say $s'^2+1=y^2$ which forces $s'=0$, a contradiction.

\smallskip
    \textbf{Case 3: \(r_0=2\) and \(r_0<r_1\).} Then \(T_{\alpha}\cong \PSp_4(q_0).2\) and \(T_{\ell}\cong \PSp_4(q_1)\), where \(q=q_0^2=q_1^{r_1}\). Applying Lemma \ref{2|P|^5>|L|^4_2222}(2) gives $|T_{\alpha}|^5<2\,|T_{\ell}|^4\,|T|$,
which gives \[\bigl(q_0^4(q_0^4-1)(q_0^2-1)\bigr)^5
<2\Bigl(q_1^4(q_1^4-1)(q_1^2-1)/{2}\Bigr)^4\cdot
q_0^8(q_0^8-1)(q_0^4-1)/2.\] Simplifying and substituting \(q_0^2=q_1^{r_1}\), we obtain
\begin{equation}\label{Case_3}
2^4\cdot \frac{q_1^{6r_1}(q_1^{2r_1}-1)^3(q_1^{r_1}-1)^5}{q_1^{2r_1}+1}
<\bigl[q_1^4(q_1^4-1)(q_1^2-1)\bigr]^4
<q_1^{40}.
\end{equation}
Since \(x^2+1<2x^2\) and \(2^3(x^2-1)^3(x-1)^5>x^{11}\) for \(x=q_1^{r_1}>3\), inequality \eqref{Case_3} implies $q_1^{15r_1}<q_1^{40}$. Therefore  \(r_1<3\), contradicting the assumption \(2=r_0<r_1\).

\smallskip
    \textbf{Case 4: $r_0=r_1=2$. } Then $T_\a \cong T_\ell \cong \PSp_4(q_0).2$ and $q=q_0^2$. From Table \ref{tab:maximal_subgroups}, we have 
$
|\PA|=|\mathcal{L}|=q_0^4(q_0^4+1)(q_0^2+1)/2=(s+1)(s^2+1).
$
 Proceeding as in the proof of Lemma~\ref{type_d}, we deduce that   $q_0^4 \mid (s^2+1)$  and  so $q_0\equiv 1\pmod{4}$. Then $q_0^4+1\equiv q_0^2+1\equiv 2 \pmod{4}$. Therefore $|\PA|_2=(q_0^4(q_0^4+1)(q_0^2+1)/2)_2=2$. On the other hand,  \(s\) must be odd and so $((s+1)(s^2+1))_2\geq 4$, which contradicts $|\PA|_2=2$. 

Thus all possibilities for type (j) have been ruled out. This completes the proof.
 \end{proof}

\begin{lemma}\label{lam:type_k}
    The subgroup $T_\a$ cannot be of type (k).
\end{lemma}
\begin{proof}
The argument is analogous to that of type (j). Suppose that $T_\a \cong \PSp_4(q_0)$ and $T_\ell \cong \PSp_4(q_1)$, where $q=2^f\geq 4$, $q=q_0^{r_0}=q_1^{r_1}$ and $r_0\leq r_1$ are primes. 

 \smallskip
    \textbf{Case 1: $r_0>2$.} Then $r_0$ and $r_1$ are odd primes.  Let \(g \in T\) be an element of order \(4\).  By Lemma \ref{Order_4_p=2}, $C_T(g)$ is abelian and $|C_T(g)|=2q^2$. It follows from \cite[Proposition 4.1 (1)]{gon2017unipotent_class} that all conjugacy classes of unipotent elements in $T$ are represented in $T_\a$ and $T_\ell$.  Hence, for every $T$-conjugacy class $C$ of elements of order $4$, the intersections $C\cap T_\alpha$ and $C\cap T_\ell$ are single conjugacy classes in $T_\alpha$ and $T_\ell$, respectively. After replacing \(T_\alpha\) by a suitable conjugate, we may assume \(g \in T_\alpha \cap T_\ell\), so that  
\(g^T \cap T_\alpha = g^{T_\alpha}\) and \(g^T \cap T_\ell = g^{T_\ell}\). Hence $C_T(g)$ acts transitively on $\mathcal{S}_g$ by Lemma \ref{CT(g) acts transitively}. Consequently, $|\PA_g|=\frac{|C_T(g)|}{|C_{T_\a}(g)|}=\frac{q^2}{q_0^2}>2$ and $|\LA_g|=\frac{q^2}{q_1^2}>2$. Lemma \ref{is a subquadrangle}(2) then forces \(C_T(g)\) to be  nonabelian, contradicting the fact that $C_T(g)$ is abelian.

 \smallskip
    \textbf{Case 2: $r_0=2$ and $r_0<r_1$.} Then $q=q_0^2=q_1^{r_1}$. The inequality in Lemma \ref{2|P|^5>|L|^4_2222}(2) yields $|T_{\alpha}|^5<2\,|T_{\ell}|^4\,|T|$, which gives
\begin{equation}\label{type_k_eq}
   \frac{q_1^{6r_1}(q_1^{2r_1}-1)^3(q_1^{r_1}-1)^5}{q_1^{2r_1}+1}<2 \bigl(q_1^4(q_1^4-1)(q_1^2-1)\bigr)^4<2q_1^{40}. 
\end{equation}
As $x^2+1<2x^2$ and $(x^2-1)^3(x-1)^5>x^{11}/2$ for $x=q_1^{r_1}\geq8$, the inequality (\ref{type_k_eq}) can be reduced to $q_1^{15r_1}<8q_1^{40}\leq q_1^{43}$. Therefore $r_1<3$, contradicting the assumption $2<r_1$.

 \smallskip
  \textbf{Case 3: $r_0=r_1=2$.} From Table \ref{tab:maximal_subgroups}, we have $|\PA|=|\mathcal{L}|=q_0^4(q_0^4+1)(q_0^2+1)=(s+1)(s^2+1)$. Since $q_0$ is even, $s$ is odd and $(s^2+1)_2=2$. Hence $(s+1)_2=q_0^4/2$. Consequently, for $q_0\geq 4$, one checks that \[(s+1)(s^2+1)\geq \frac{q_0^4}{2}\left(\frac{q_0^8}{4}-q_0^4+2\right)>q_0^4(q_0^4+1)(q_0^2+1)=|\PA|,\] a contradiction. If $q_0=2$, then $|\PA|=(s+1)(s^2+1)=1360$, which has no integer solution for $s$. 

  All possibilities have been ruled out, so \(T_\alpha\) cannot be of type (k). 
\end{proof}

\begin{lemma}\label{p_1_mod 4}
  Suppose $T_\a$ has type (l), (m), (o) or (r). Then $p\equiv 1\pmod{4}$.  
\end{lemma}
\begin{proof}
For each of these types, \(q = p^f\) is odd and the order \(|\mathcal{P}|\) has the form \(q^a \cdot M\), where \(a = 4\) for types (l), (m), (o) and \(a = 3\) for type (r), while \(M\) is an integer coprime to \(p\).  Since \(\gcd(s+1,s^2+1)\) divides \(2\) and \(p\) is odd,  \(q^a\) must divide either \(s+1\) or \(s^2+1\).
If \(q^a \mid (s+1)\), then \(s+1 \ge q^a\) and consequently
$(s+1)(s^2+1) \ge q^a\bigl((q^a-1)^2+1\bigr)$. For each type, one verifies directly that for all $q$, $(q^a-1)^2+1>M$. For example, in type (l) we have \(a=4\) and \(M = (q^4-1)(q^2-1)/3840\); the inequality reduces to
$q^8 - 2q^4 + 2 >(q^4-1)(q^2-1)/3840$, which clearly holds for all \(q \ge 3\). The other types are checked similarly. Thus the case \(q^a \mid (s+1)\) leads to a contradiction.
Therefore \(q^a\) must divide \(s^2+1\). In particular, \(s^2 \equiv -1 \pmod{p}\), which forces \(p \equiv 1 \pmod{4}\).
\end{proof}

\begin{lemma}
    The subgroup $T_\a$ cannot be of type (l).
\end{lemma}
\begin{proof}
    Suppose, for a contradiction, that $T_\a$ is of type (l). From Table \ref{tab:maximal_subgroups}, $T_\a \cong T_\ell \cong \hat{}~2_-^{1+4}.\mathrm{S}_5 \cong   2^4:\mathrm{S}_5$, and \[
|\mathcal{P}| = |\mathcal{L}|  = (s+1)(s^2+1)= \frac{q^4(q^4-1)(q^2-1)}{3840},
\]
where $q = p \equiv \pm 1 \pmod{8}$. It follows from Lemma \ref{p_1_mod 4} that $q=p\equiv 1 \pmod{8}$.
    
    A Magma computation~\cite{MAGMAbosma1994handbook} shows that the group $2^4:\mathrm{S}_5$ has five conjugacy classes of involutions, with class sizes $5,10,20,60,$ and $60$. On the other hand, by Lemma \ref{Conj_Order=2}, $T$ has exactly two conjugacy classes of involutions.  Let $g \in T$ be a representative of the conjugacy class $Y_1$ in Lemma \ref{Conj_Order=2}, so that $|g^T|=q^2(q^2+1)/2$. 
    A Magma computation \cite{MAGMAbosma1994handbook} shows that  the group $2^{1+4}_-.\mathrm{S}_5$ contains exactly $51$ involutions. Since $T_\a$ and $T_\ell$ contain $Z(\Sp_4(q))$,  Lemma \ref{Conj_Order=2}(2) yields $|g^T \cap T_\a|=25$ and $|g^T \cap T_\ell|=25$. Equivalently, the $T$-class $Y_1$ splits in $T_\a$ and in $T_\ell$ into two involution classes of sizes $5$ and $20$. After conjugating in \(T\) if necessary, we may assume that \(g \in T_\alpha \cap T_\ell\) and that \(|g^{T_\alpha}| = 20\). Then $|C_{T_\a}(g)|=|T_\a|/|g^{T_{\a}}|=96$ and Lemma \ref{Fix_points_number} gives \[
|\mathcal{P}_g| = |\mathcal{L}_g| = \frac{|\mathcal{P}| \cdot |g^T \cap T_\alpha|}{|g^T|}
= \frac{5 q^2 (q^2-1)^2}{384} > 2.
\]
 
    Since $|g^T \cap T_\a|\neq|g^{T_\a}|$, Lemma \ref{CT(g) acts transitively} implies that $C_T(g)$ acts intransitively on $\PA_g$ and $\LA_g$. We next show that $C_T(g)$ has exactly two orbits on $\PA_g$ (and similarly on $\LA_g$). The orbit $\a^{C_T(g)}$ has length $|C_T(g)|/|C_{T_{\a}}(g)|=q^2(q^2-1)^2/96$. For any other point $\beta \in \mathcal{P}_g$, write $\beta = \alpha^h$ with $h \in T$. Then $g \in T_{\beta}$ lies in an involution class of size $5$ or $20$ inside $T_{\beta} \cong T_{\a}$, so $|\beta^{C_T(g)}|$ is either $\tfrac{q^2(q^2-1)^2}{384}>2$ or $\tfrac{q^2(q^2-1)^2}{96}>2$. 
Let \(x\) and \(y\) be the numbers of \(C_T(g)\)-orbits on \(\mathcal{P}_g\) of lengths \(\frac{q^2(q^2-1)^2}{96}\) and \(\frac{q^2(q^2-1)^2}{384}\), respectively. Then \[|\PA_g|=\frac{5q^2(q^2-1)^2}{384}=x\cdot \frac{q^2(q^2-1)^2}{96}+y \cdot \frac{q^2(q^2-1)^2}{384},\] which simplifies to $4x + y = 5$. Since $x>0$, we obtain $x=y=1$. Hence $C_T(g)$ has exactly two orbits on $\PA_g$, say $\PA_1$, $\PA_2$, of sizes $\tfrac{q^2(q^2-1)^2}{96}$ and $\tfrac{q^2(q^2-1)^2}{384}$, respectively. Applying the same counting argument to $\mathcal{L}_g$  shows that $C_T(g)$ likewise has two orbits on $\mathcal{L}_g$, which we denote by $\mathcal{L}_1$ and $\mathcal{L}_2$, with $|\mathcal{L}_1| = |\mathcal{P}_1|$ and $|\mathcal{L}_2| = |\mathcal{P}_2|$. As $|\LA_1|>|\PA_2|$ and $|\PA_1|>|\LA_2|$, it follows from Lemma \ref{lem:two-orbits-P-and-L-subqua} that $\mc{S}_g$ is a subquadrangle of order $s'$. Therefore, 
    \begin{equation}\label{type_l_Pg_number}
        |\PA_g|=\frac{5q^2(q^2-1)^2}{384}=(s'+1)({s'}^2+1)
    \end{equation}


   Set $X = 30s'+10$ and $Y = 75q(q^2-1)/4$. Substituting into (\ref{type_l_Pg_number}) yields the equation
     \begin{equation}\label{typq_l_eq}
         Y^2=X^3+600X+20000
     \end{equation}
     Every positive integer solution of (\ref{type_l_Pg_number}) gives a positive integer solution of (\ref{typq_l_eq}). The cubic in \eqref{typq_l_eq} is nonsingular, since its Weierstrass discriminant is $\Delta=-186624000000\neq 0$. Thus, \eqref{typq_l_eq} defines an elliptic curve. Using Magma \cite{MAGMAbosma1994handbook}  with the command \texttt{E := EllipticCurve([0, 0, 0, 600, 20000])}, we find that the positive integral solutions of \eqref{typq_l_eq} are $(X,Y)=(20,200),(25,225),(7180,608400)$. Among these, only $(X,Y)=(7180,608400)$ satisfies $30 \mid (X-10)$.  However, in this case, $75q(q^2-1)/4 = 608400$ has no integer solution  for $q$, a contradiction. This completes the proof.
   \end{proof}

\begin{lemma}
    The subgroup $T_\a$ cannot be of type (m),(o) or (r).
\end{lemma}
\begin{proof}
Suppose, for a contradiction, that $T_\a$ is isomorphic to $2^4:\mathrm{A}_5$, $\mathrm{A}_6$ or $\PSL_2(q)$ from Table \ref{tab:maximal_subgroups}. By Lemma \ref{p_1_mod 4}, we deduce that $p>3$.
It follows from \cite[Table 8.12 and Table 8.13]{bray2013maximal} that $T$ has a unique conjugacy class of subgroups isomorphic to $2^4:\mathrm{A}_5$, $\mathrm{A}_6$ and $\PSL_2(q)$. Hence, after replacing $T_\alpha$ by a suitable conjugate, we may assume $T_\alpha = T_\ell$. By Magma \cite{MAGMAbosma1994handbook}, the group $2^4:\mathrm{A}_5$ has a unique conjugacy class of elements of order $3$. 
However, by Lemma \ref{Conj_order=3}, $T$ has exactly two conjugacy classes of elements of order $3$. 
Therefore there exists an element \(g\in T\) of order \(3\) such that \(g^T\cap T_{\alpha}=g^T\cap T_\ell=\varnothing\), contradicting Lemma \ref{lem:no_2_3_elemen}. Similarly, Magma computations \cite{MAGMAbosma1994handbook} and \cite[Lemma~2.9]{feng2023finitePSL} show that $\mathrm{A}_6$ and \(\PSL_2(q)\) each have a single conjugacy class of involutions, whereas \(T\) has two such classes by Lemma \ref{Conj_Order=2}. 
Thus there exists an  involution \(g\in T\)  such that \(g^T\cap T_{\alpha}=g^T\cap T_\ell=\varnothing\), again contradicting Lemma \ref{lem:no_2_3_elemen}. Thus none of the types (m), (o) or (r) can occur.
\end{proof}

\begin{lemma}
    The subgroup $T_\a$ cannot be of type (p).
\end{lemma}
\begin{proof}
Suppose that $T_\a$ is of type (p). Then $T_\a \cong T_\ell \cong \mathrm{S}_6$ for $q=p\equiv 1,11 \pmod{12}$. From Table \ref{tab:maximal_subgroups}, $|\PA|=q^4(q^4-1)(q^2-1)/1440$. Let $g$ be an involution in the $T$-conjugacy class $Y_1$ from Lemma \ref{Conj_Order=2}. Then $|g^T|=q^2(q^2+1)/2$. A Magma computation \cite{MAGMAbosma1994handbook} shows that the preimage $2.\mathrm{S}_6\le \Sp_4(q)$ of $\mathrm{S}_6$ contains  $31$ involutions. Since $T_\a$ and $T_\ell$ both contain $Z(\Sp_4(q))$, Lemma \ref{Conj_Order=2}(2) yields $|g^T \cap T_\a|=15$ and $|g^T \cap T_\ell|=15$. Since conjugacy classes in $\mathrm{S}_6$ are determined by cycle types, $\mathrm{S}_6$ has exactly three conjugacy classes of involutions, of sizes $15$, $45$, and $15$.  Thus, for the $T$-conjugacy class $Y_1$ in Lemma \ref{Conj_Order=2}, the intersections $Y_1\cap T_\alpha$ and $Y_1\cap T_\ell$ are single conjugacy classes in $T_\alpha$ and $T_\ell$, respectively. After  replacing \(T_\alpha\) by a suitable conjugate, we may assume \(g \in T_\alpha \cap T_\ell\). Hence $g^T \cap T_\a=g^{T_\a}$ and $g^T \cap T_\ell=g^{T_\ell}$, so $C_T(g)$ acts transitively on $\PA_g$ and $\LA_g$ by Lemma \ref{CT(g) acts transitively}. And Lemma \ref{Fix_points_number} yields $|\PA_g|=|\LA_g|=q^2(q^2-1)^2/48>2$. By Lemma \ref{is a subquadrangle}, $(\PA_g ,\LA_g)$ is a subquadrangle of order $s'$. Therefore,
\begin{equation}\label{type_p}
    q^2(q^2-1)^2/48=(s'+1)({s'}^2+1).
\end{equation}
Set $X=3s'+1$ and $Y=3q(q^2-1)/4$. Then (\ref{type_p}) reduces to
\begin{equation}\label{typq_p_eq}
    Y^2=X^3+6X+20.
\end{equation}
Since its Weierstrass discriminant is $\Delta=-186624\neq 0$, the cubic in \eqref{typq_p_eq} is nonsingular. Thus, \eqref{typq_p_eq} defines an elliptic curve. Using Magma \cite{MAGMAbosma1994handbook} with the command \texttt{E:=EllipticCurve([0,0,0,6,20])}, one finds that (\ref{typq_p_eq}) has exactly one integral solution $(X,Y)=(-2,0)$, which gives $s'=-1$, a contradiction. Hence $T_\a$ cannot be of type (p).
\end{proof}

\begin{lemma}\label{type_q_s}
The subgroup \(T_{\alpha}\) cannot be of type (q) or (s).
\end{lemma}
\begin{proof}
Suppose first that \(T_{\alpha}\) is of type (q). Then \(|\PA|=(s+1)(s^2+1)=54880\), which has no solution in \(s\), a contradiction. Now suppose that \(T_{\alpha}\) is of type (s), where \(q=p^f\) is even and \(f\) is odd. By \eqref{number_point}, \(|\PA|=q^2(q+1)^2(q-1)=(s+1)(s^2+1)\) is even. Hence \(s\) is odd, so \((s^2+1)_2=2\) and \((s+1)_2=q^2/2\). Therefore, for \(q>8\),
\[(s+1)(s^2+1)\ge \tfrac{q^2}{2}\bigl(\tfrac{q^4}{4}-q^2+2\bigr)>q^2(q+1)^2(q-1)=|\PA|,\]
a contradiction. It remains to consider \(q=8\). In this case \(|\PA|=(s+1)(s^2+1)=36288\), which again has no solution for \(s\). This completes the proof.
\end{proof}

\subsection{The case when $\Ta$ and $T_\ell$ are of different types}\label{section:differnet-type}
In this section, we assume that \(T_\alpha\) and \(T_\ell\) are of different types in Table~\ref{tab:maximal_subgroups}. By point-line duality, we may further assume that $T_\alpha$ precedes $T_\ell$ alphabetically in the table. For each pair \((T_\alpha, T_\ell)\), we obtain \(|\PA|\) and \(|\LA|\) from the last column of Table \ref{tab:maximal_subgroups} by \eqref{number_point} and \eqref{number_line}, respectively. We then apply the inequality in Lemma~\ref{2|P|^5>|L|^4_2222}(2) to obtain restrictions on \(q\), and analyze the remaining possibilities.

\begin{lemma}\label{W(3,q)-q-odd}
Suppose that $T_\a$ is of type (a) or (b). Then up to duality, $\mc{S}$ is the classical symplectic quadrangle $W(3,q)$ for odd $q\ge 3$. 
\end{lemma}

\begin{proof}
Suppose that $T_\alpha$ is of type (a) or (b). Then from Table \ref{tab:maximal_subgroups}, we have $|\mathcal{P}| = (q+1)(q^2+1)$.  For each possible type of $T_\ell$, we first use Lemma~\ref{2|P|^5>|L|^4_2222}(2) to restrict $q$. Since the arguments for different types are similar, we present the details only for type (j) and omit the rest. Suppose that $T_\ell$ is of type (j). Then $
|\mathcal{L}| = \frac{q^4(q^4-1)(q^2-1)}{q_0^4(q_0^4-1)(q_0^2-1)\cdot\gcd(2,r)}$, where $q=q_0^r$ and $r$ is prime. Note that $\frac{q^a-1}{q_0^a-1}>q_0^{a(r-1)}$ holds for any positive integer $a$. It follows that $|\LA|> q_0^{10(r-1)}/2$. On the other hand, using $|\mathcal{P}| < 4q_0^{3r}$  and Lemma \ref{2|P|^5>|L|^4_2222}(2), we obtain
    $$
    q_0^{10(r-1)}/2\leq |\LA| < 2^{1/4}|\PA|^{5/4}<2^{11/4}q_0^{15r/4}.
    $$
Since $q_0\geq 3$ and $r\geq 2$, this would imply $3^{5/2}<2^{15/4}$, which is false. Hence $T_\ell$ cannot be of type (j).

    For the remaining possible $T_\ell$ type, by the inequality in Lemma \ref{2|P|^5>|L|^4_2222}(2), we obtain the following restrictions: $q$ is odd if $T_\ell$ is of type (b); $q\leq 32$ if $T_\ell$ is of type (e) or (g). For the latter two types, a direct computer check shows that there are no integers $s,t$ satisfying \eqref{number_point} and \eqref{number_line}. Consequently, the only remaining possibility is that $T_\alpha$ is of type (a) and $T_\ell$ is of type (b) for odd $q$. In this case, we have $|\PA|=|\LA|=(q+1)(q^2+1)$, which gives $s=t=q$. By the geometric description in \cite[Table 2.3]{bray2013maximal}, the point set $\PA$ and line set $\LA$ can be identified with the sets of $1$-dimensional and $2$-dimensional totally isotropic subspaces of the $4$-dimensional symplectic space, respectively. By Witt's Lemma \cite[Theorem 7.4]{taylor1992geometry}, $T_\alpha$ has exactly two orbits on $\LA$, consisting of the $2$-dimensional totally isotropic subspaces containing the $1$-dimensional subspace corresponding to $\alpha$, and those not containing it, of sizes $q+1$ and $q^2(q+1)$, respectively. Since the $q+1$ lines incident with $\alpha$ form a $T_\alpha$-invariant subset of $\LA$, they are precisely the lines in the first orbit. Thus the incidence relation is given by inclusion, and therefore $\mathcal{S}\cong W(3,q)$ for odd $q$.
    
\end{proof}


\begin{lemma}
    The subgroup $T_\a$ cannot be of type (c).
\end{lemma}
\begin{proof}
    Suppose, for a contradiction, that $T_\alpha$ is of type (c). Then $|\PA|=(q^2+1)(q+1)^2$, where   $q$ is even. We first consider the case where $T_\ell$ is of type (k). Then $|\LA|=\tfrac{q^4(q^4-1)(q^2-1)}{q_0^4(q_0^4-1)(q_0^2-1)}$, where $q=q_0^r$ and $r$ is prime. Note that $|\PA|<2q_0^{4r}$ and $|\LA|\geq q_0^{10(r-1)}$. Lemma \ref{2|P|^5>|L|^4_2222}(2) gives $2^{5r-10}\leq q_0^{5r-10}<2^{3/2}$. Hence  $r=2$. Then $|\LA|=q^2(q^2+1)(q+1)$ and therefore $\frac{s+1}{t+1}=\frac{|\PA|}{|\LA|}=\frac{q+1}{q^2}$. We   write $s+1=(q+1)k$ and $t+1=q^2k$ for some positive integer $k$. Applying Lemma \ref{2|P|^5>|L|^4_2222}(1) to $|\mathcal{P}|$ yields \[(q^2+1)(q+1)^2=|\PA|>(s+1)^2(t+1)/2=q^2(q+1)^2k^3/2,\] so $k^3 < 2 + 2/q^2 < 3$, whence $k=1$. Consequently $(s,t) = (q, q^2-1)$. Substituting into $|\mathcal{P}|$ gives
$(q+1)(q^3-q+1) = (q^2+1)(q+1)^2$, which simplifies to $q(q+2)=0$, impossible. Hence $T_\ell$ cannot be of type (k).

    We now apply the bound in Lemma~\ref{2|P|^5>|L|^4_2222}(2) to the remaining possibilities for $T_\ell$.  This restricts $T_\ell$ to types (e), (g), (n) or (s). Suppose that $T_\ell$ is of type (e). Then $|\LA|=q^2(q^2+1)/2$ and $\frac{s+1}{t+1}=\frac{|\PA|}{|\LA|}=\frac{2(q+1)^2}{q^2}$. Since $q$ is even, we have $\gcd((q+1)^2,q^2/2)=1$. Hence $s+1\geq (q+1)^2$ and $t+1\geq q^2/2$.  Applying Lemma \ref{2|P|^5>|L|^4_2222}(1) gives \[q^2(q^2+1)/2=|\LA|>(t+1)^2(s+1)/2\geq q^4(q+1)^2/8,\] impossible for any $q \ge 4$. Therefore $T_\ell$ cannot be of type (e). The remaining types (g), (n) and (s) can be ruled out by an analogous argument, and we omit the details.
\end{proof}

\begin{lemma}\label{diff_type_d}
    The subgroup $T_\a$ cannot be of type (d) or (h).
\end{lemma}
\begin{proof}
We argue by contradiction.

\smallskip
    \textbf{Case 1: $T_\a$ is of type (d).} Then $|\PA|=q^3(q^2+1)(q+1)/2$, where $q\geq5$ is odd. We first consider the case where $T_\ell$ is of type (j). Then $|\LA|=\tfrac{q^4(q^4-1)(q^2-1)}{q_0^4(q_0^4-1)(q_0^2-1)\cdot\gcd(2,r)}$. By Lemma \ref{2|P|^5>|L|^4_2222}(2), we deduce that $q_0^{10(r-1)}/2\leq |\LA|< 2^{1/4}|\PA|^{5/4}<2^{1/4}q_0^{15r/2}$, which implies that $r=2$ or $3$.
    
    Assume first that $r=2$. Then $|\LA|=q^2(q^2+1)(q+1)/2$ and  $\frac{|\PA|}{|\LA|} = q$. We may assume that $s+1=qk$ and $t+1=k$ for some positive integer $k$.  If $p\nmid (st+1)$, then $(s+1)_p=|\PA|_p=q^3$ and $(t+1)_p=|\LA|_p=q^2$. By Lemma \ref{2|P|^5>|L|^4_2222}(1), \[|\PA|=q^3(q^2+1)(q+1)/2>(s+1)^2(t+1)/2=q^2k^3/2>q^8/2 \] since $k \ge q^2$ from $(s+1)_p = q^3$,  a contradiction. Hence $p\mid st+1$ and $k<q^2$.  Write $k=mq+a$ with    $0 \leq m,a < q$. As $\gcd(s+1,t+1,st+1)\mid 2$ and $q$ is odd, we have $(st+1)_p=q^2$.
     Reducing the expression for $st+1$ modulo $q$   gives $0\equiv st+1\equiv 2-a\pmod q$, so $a=2$. Substituting $a=2$ into $st+1$ and reducing modulo $q^2$ yields $0\equiv st+1\equiv (2-m)q\pmod{q^2}$, so $m=2$. Finally, substituting $k = 2q+2$ into $|\PA|=(s+1)(st+1)$, we  deduce   that  $q^2 - 16q - 23 = 0$. This quadratic has no integer solution for $q \ge 5$, a contradiction. Therefore $r=2$ cannot occur.

    Next assume that $r=3$. Define  $A:=q_0(q_0^2+1)(q_0+1)/{2d}$, $B:=(q_0^4+q_0^2+1)(q_0^2+q_0+1)/d$ and $g:=q_0^8(q_0^2-q_0+1)(q_0^4-q_0^2+1)d$, where $d=\gcd(3,q_0+1)$. Then from  Table \ref{tab:maximal_subgroups}, $|\PA|=Ag$, $|\LA|=Bg$ and $\frac{s+1}{t+1}=\frac{A}{B}$. Since $\gcd(A,B)=1$, we may write $s+1=Ak$ and $t+1=Bk$ for some positive integer $k$, and then $st+1=g/k$. We can rule out the cases $q_0<7$ by solving the equations in (\ref{number_point}) and (\ref{number_line}). So we assume $q_0\geq 7$. If $k>q_0^2$, then by Lemma \ref{2|P|^5>|L|^4_2222}(1), we have \[Ag=|\PA|>(s+1)^2(t+1)/2=A^2Bk^3/2>A^2Bq^2/2,\] which is impossible for $q_0\geq 7$. Hence $k<q_0^2$. Write $k=mq_0+a$ with $0\le m,a<q_0$. Moreover, we deduce that $(st+1)_{q_0}=(g/k)_{q_0}\geq q_0^6$. Note that $st+1=ABk^2-(A+B)k+2$, and we have $st+1\equiv 0\pmod{q_0}$ and $st+1\equiv 0\pmod{q_0^2}$. Reducing $st+1$ modulo $q_0$ yields $a\equiv 2d\pmod{q_0}$. If $d=1$, then $a=2$. If $d=3$, then $a=6$ for $q_0>5$, and $a=1$ for $q_0=5$. Reducing $st+1$ modulo $q_0^2$ then gives $m=q_0-1$ when $d=1$; $m=q_0-3$ when $d=3$ and $q_0>5$; $m=3$ when $(q_0,d)=(5,3)$. Hence $k=q_0^2-q_0+2$, $q_0^2-3q_0+6$, or $k=16$, respectively. In each case $k$ is even, whereas $g=q_0^8(q_0^2-q_0+1)(q_0^4-q_0^2+1)d$ is odd. This contradicts the fact that $k\mid g$. Therefore $r=3$ is impossible.

\smallskip
    We now apply the bound in Lemma \ref{2|P|^5>|L|^4_2222}(2) to the remaining possibilities for $T_\ell$.  This restricts $T_\ell$ to types (l)–(q) with $q\le 23$, type (h) or type (r). For types from (l) to (q), a computer verification shows that no pair $(s,t)$ satisfies the equations (\ref{number_point}) and (\ref{number_line}). Thus it remains to consider types (h) and (r). 
    
\smallskip
If $T_\ell$ is of type (h), then $|\LA|=q^3(q^2+1)(q-1)/2$ and $\frac{s+1}{t+1}=\frac{|\PA|}{|\LA|}=\frac{q+1}{q-1}$.
Since $q$ is odd, $\gcd(q+1,q-1)=2$, so we may write $s+1=k(q+1)/2$ and $t+1=k(q-1)/2$ for some positive integer $k$.
If $k\leq q$, then
\[q^3(q^2+1)(q+1)/2=|\PA|<(s+1)^2(t+1)=(q+1)^2(q-1)k^3/8\leq q^3(q+1)^2(q-1)/8,\]
which is impossible. Hence $k>q$.
By Lemma~\ref{parameters}(2), $s+t=qk-2$ divides $st(s+1)(t+1)=st(q^2-1)k^2/4$.
A routine gcd computation yields $\gcd(s,t)=1$, and $\gcd(s+t,k)=\gcd(qk-2,k)$ divides $2$.
Consequently  $s+t$ divides $q^2-1$.
But $k>q$ gives $q^2-2<qk-2=s+t\leq q^2-1$, so $qk-2=q^2-1$, impossible. Thus $T_\ell$ cannot be of type (h).

    \smallskip
 If $T_\ell$ is of type (r), then $|\LA|=q^3(q^4-1)$ and $t+1=2(q-1)(s+1)$. Set $s+1=k$, so $t+1=2(q-1)k$ and $s+t=(2q-1)k-2$. By Lemma \ref{parameters}(3), $s+t \mid st(s+1)(t+1)$. Moreover,
$\gcd(s,t)=\gcd(k-1,2(q-1)k-1)$ divides $2q-3$,
$\gcd(s+t,s+1)=\gcd(2,k)$ and $\gcd(s+t,t+1)=\gcd(4(q-1),k-2)$.
Consequently, $s+t \mid 8(q-1)(2q-3)^2$. Denote $N = 8(q-1)(2q-3)^2$ and write $N = (s+t)a$ for some positive integer $a$. Reducing $N$ and $s+t$ modulo $2q-1$ gives
\[
N \equiv -16 \pmod{2q-1},\qquad s+t \equiv -2 \pmod{2q-1},
\]
hence $-2a \equiv -16 \pmod{2q-1}$, i.e., $a \equiv 8 \pmod{2q-1}$.  Therefore either $a = 8$ or $a \ge 2q+7$. If $a=8$, then $s+t=(q-1)(2q-3)^2=(2q-1)k-2$, whence $k=2q^2-7q+7$. By Lemma \ref{2|P|^5>|L|^4_2222}(1), we have 
$$
q^3(q^4-1)=|\LA|>(t+1)^2(s+1)/2=2(q-1)^2(2q^2-7q+7)^3
$$
which is false for   $q\geq 5$. If $a\geq 2q+7$, then from $s+t\leq N/(2q+7)$ we deduce that   $k=(s+t+2)/(2q-1)<8q-7$. Therefore,
$$
q^3(q^4-1)=|\LA|\leq(t+1)^2(s+1)=4(q-1)^2k^3<4(q-1)^2(8q-7)^3,
$$
which is false for $q\geq 47$. For $q<47$, we check each case by computer that there is no solution $(s,t)$ for equations (\ref{number_point}) and (\ref{number_line}). Thus $T_\ell$ cannot be of type (r).

\smallskip
    \textbf{Case 2: $T_\a$ is of type (h).}
For odd $q\ge 5$, we have $|\PA|=\frac{1}{2}q^{3}(q^{2}+1)(q-1)$.
Applying Lemma~\ref{2|P|^5>|L|^4_2222}(2) and arguing as Case 1, we find that $T_\ell$ can only be of type (l) to (q)  with $q\le 19$, or of type (r).
All  these possibilities are excluded by the same computations as in Case 1 (including the same computer verification for types (l)--(q)), and we omit the routine details.
\end{proof}

\begin{lemma}\label{H(3,4)}
 Suppose that $T_\alpha$ is of type (e) or (g).
Then the only possibility is $q=3$, with $T_\alpha$ of type (e)
and $T_\ell$ of type (m), in which case $\mc{S}\cong H(3,4)$.
\end{lemma}
\begin{proof}
We consider the two possibilities for $T_\alpha$ separately.

\smallskip
\textbf{Case 1: $T_\alpha$ is of type (e).} Then $|\mathcal{P}| = q^2(q^2+1)/2$. Applying Lemma \ref{2|P|^5>|L|^4_2222}(2) and comparing with the possible values of $|\mathcal{L}|$ from Table~\ref{tab:maximal_subgroups},  we find that $T_\ell$ can only be of type (g), of type (n), or of type (m) with $q=3$. 

    Assume that $T_\ell$ is of type (g). Then $|\LA|=q^2(q^2-1)/2$ and $\frac{s+1}{t+1}=\frac{|\PA|}{|\LA|}=\frac{q^2+1}{q^2-1}$. Since $\gcd(q^2-1,q^2+1)$ divides $2$, we have $s+1\geq (q^2+1)/2$ and $t+1\geq (q^2-1)/2$. By Lemma~\ref{2|P|^5>|L|^4_2222}(1), it follows that $q^2(q^2+1)/2=|\PA|>(q^2+1)^2(q^2-1)/16$, which is impossible for $q\geq 3$.

    Assume that $T_\ell$ is of type (n). Then $|\LA|=q^2(q^2+\epsilon)/2$ with $\epsilon=\pm1$. The case $\epsilon=-1$ is excluded by the same argument as in type (g). Hence we may assume that $\epsilon=+1$, in which case $|\mathcal{P}| = |\mathcal{L}|$ and consequently $s = t$. Equation (\ref{number_point}) then reduces to $(s+1)(s^2+1) = q^2(q^2+1)/2$, which has no solution in integers $s>1$ by the proof of Lemma~\ref{type_e_f_g_i_n}. 

     Now assume that $T_\ell$ is of type (m) with $q=3$. Then $|\PA|=(s+1)(st+1)=45$ and $|\LA|=(t+1)(st+1)=27$. Since $\mc{S}$ is thick, $s,t\geq 2$, hence $st+1\geq 5$. Moreover, $st+1$ divides $\gcd(45,27)=9$. Therefore, $st+1=9$. It follows that $(s,t)=(4,2)$. By \cite[5.3.2 and 3.2.3]{PayneThas2009}, the unique generalized quadrangle of order \((4,2)\) is \(\mathrm H(3,4)\). Therefore, \(\mathcal S\cong\mathrm H(3,4)\).

\smallskip
\textbf{Case 2: $T_\a$ is of type (g).}
Lemma~\ref{2|P|^5>|L|^4_2222}(2) implies that either
$T_\ell$ is of type (m) with $q=3$, or $T_\ell$ is of type (n)
with $q$ even. The former possibility can be excluded, since
there are no integers $s,t>1$ satisfying both
(\ref{number_point}) and (\ref{number_line}). The latter
possibility can be excluded by the same argument as in the
preceding case for type (n), so we omit the details. This completes the proof.
\end{proof}

\begin{lemma}\label{lem:T_a-f-i-and-T_l-k}
    Suppose that $T_\a$ is of type (f) or (i). Then $T_\ell$ cannot be of type (k).
\end{lemma}
\begin{proof}
    Suppose, for a contradiction, that $T_\ell$ is of type (k). Then $q$ is even and
    $T_\ell \cong \PSp_4(q_0)$ with $q=q_0^r$ for some prime $r$.

    \smallskip
    \textbf{Case 1: $T_\a$ is of type (f).}
    Then $T_\a \cong \mathrm{C}_{q+\epsilon}^2:\mathrm{D}_8$ for $\epsilon=\pm1$.
    If $r = 2$, then applying Lemma~\ref{2|P|^5>|L|^4_2222}(2) together with the formulas for $|\mathcal{P}|$ and $|\mathcal{L}|$ from Table~\ref{tab:maximal_subgroups} yields a contradiction. Hence $r > 2$. Since $q$ is even, $\mathrm{C}_{q+\epsilon}^2:\mathrm{D}_8$ has exactly one conjugacy class of elements of order $4$, of size $2(q+\epsilon)^2$.
    Take $g \in T_\a$ of order $4$. Then $g^T \cap T_\a=g^{T_\a}$. As in the proof of Lemma \ref{lam:type_k}, all conjugacy classes of elements of order $4$ in $T$ are represented in $T_\ell$.
    Thus,  after conjugating in $T$ and replacing $\alpha$ by a suitable conjugate if necessary, we may assume $g \in T_\alpha \cap T_\ell$.
    Then $g^T\cap T_\ell=g^{T_\ell}$.

    Hence $C_T(g)$ acts transitively on both $\PA_g$ and $\LA_g$ by Lemma \ref{CT(g) acts transitively}.
    In particular,
    $|\PA_g|=\frac{|C_T(g)|}{|C_{T_\alpha}(g)|}=q^2/2>2$
    and
    $|\LA_g|=\frac{q^2}{q_0^2}>2$.
    By Lemma \ref{is a subquadrangle}, $\mc{S}_g$ is a subquadrangle, and \(C_T(g)\) is nonabelian.
    This contradicts Lemma \ref{Order_4_p=2}, which shows that $C_T(g)$ is abelian.
    Hence this case cannot occur.
    
    \smallskip
    \textbf{Case 2: $T_\a$ is of type (i).}
    Then $T_\a \cong \mathrm{C}_{q^2+1}:\mathrm{C}_4$ and we deduce that $r>2$ as in Case 1.
    By computation, $\mathrm{C}_{q^2+1}:\mathrm{C}_4$ has exactly two conjugacy classes of elements of order $4$, each of size $q^2+1$.
    Take $g \in T_\a\cap T_\ell$ of order $4$.
    Then $g^T \cap T_\ell =g^{T_\ell}$ as in the proof of Lemma \ref{lam:type_k}. And $g^T \cap T_\a=g^{T_\a}$ or $g^T \cap T_\a$ consists of all the elements of order $4$ in $T_\a$.

    If $g^T \cap T_\a=g^{T_\a}$, then we obtain a contradiction exactly as in Case~1.
    Therefore,
    $|g^T \cap T_\a|=2(q^2+1)$,
    and then by Lemma \ref{Fix_points_number} we have $|\PA_g|=q^2$.
    On the other hand, each orbit of $C_T(g)$ on $\PA_g$ has size $\frac{|C_T(g)|}{|C_{T_\alpha}(g)|}=q^2/2>2$.
    Hence $C_T(g)$ has exactly two orbits on $\PA_g$, of equal size.
    By Lemma \ref{lem:P-two-orbit-quadra}, $\mc{S}_g$ is a subquadrangle of order $(s',t')$ for some positive integers $s'$ and $t'$. Write $q_0=2^m$ with $m\geq 1$, so that $q=2^{mr}$.
    Thus
    \begin{equation}\label{type-f-k}
       |\PA_g|=(s'+1)(s't'+1)=2^{2mr},  |\LA_g|=(t'+1)(s't'+1)=2^{2m(r-1)} 
    \end{equation}
    Then $s'+1=2^{2m}(t'+1)$.
    Suppose that $t'+1=2^a$ with $a\geq 1$.
   Then $s'=2^{2m+a}-1$.
    Consequently,
    \[s't'+1=2^{2m+2a}-2^{2m+a}-2^a+2.\]
    If $a \geq 2$, then $s't'+1 \equiv 2 \pmod{4}$. Since $s't'+1$ is a power of $2$, we deduce that $s'=t'=1$. Then (\ref{type-f-k}) gives $4=2^{2mr}=2^{2m(r-1)}$, impossible.
    Hence $a=1$, so $t'=1$ and $s'+1=2^{2m+1}$.
    Substituting into (\ref{type-f-k}) gives
    $|\LA_g|=2^{2m+2}=2^{2m(r-1)}$,
    forcing $(m,r)=(1,3)$.
    Therefore, $q_0=2$ and $q=8$.
    In this case, $|\PA|=(s+1)(st+1)=4064256$ and $|\LA|=(t+1)(st+1)=1467648$ have no integer solutions  $(s,t)$.
    This completes the proof.
\end{proof}

\begin{lemma}\label{lem:type-f-i-n}
    The subgroup $T_\a$ cannot be of type (f), (i) or (n). 
\end{lemma}
\begin{proof}
    We treat the three possibilities for $T_\a$ separately.

    \medskip
    \textbf{Case 1: $T_\a$ is of type (f).}
    Then $q$ is even and $|\PA|=q^4(q-\epsilon)^2(q^2+1)/8=(s+1)(st+1)$ with $\epsilon=\pm1$.
    We only consider $\epsilon=+1$, since the case $\epsilon=-1$ is analogous.

    By Lemma \ref{lem:T_a-f-i-and-T_l-k}, $T_\ell$ cannot be of type (k).
    Applying  Lemma \ref{2|P|^5>|L|^4_2222}(2) and the same reductions as in Lemma \ref{diff_type_d},
    we find that the only remaining possibility for $T_\ell$ consistent with the bounds is type (i). Hence we  assume $T_\ell$ is of type (i). 
    Then $|\LA|=q^4(q^2-1)^2/4=(t+1)(st+1)$. Note that $\gcd(s+1,t+1,st+1)$ divides $2$, and by the equation we deduce $(t+1)_2=2(s+1)_2$. Now
    $\frac{s+1}{t+1}=\frac{|\PA|}{|\LA|}=\frac{q^2+1}{2(q+1)^2}$. 
    Since $\gcd(q^2+1,2(q+1)^2)=1$, we may assume that $s+1=(q^2+1)k$ and $t+1=2(q+1)^2k$ for some positive integer $k$. If $k\geq q/3$, then by Lemma \ref{2|P|^5>|L|^4_2222}(1),
    \[q^4(q^2-1)^2/4=|\LA|\geq (t+1)^2(s+1)/2 \geq 2q^3(q+1)^4(q^2+1)/27,\]
    which is false for all $q\geq 4$. Hence $k<q/3$. Note that $(st+1)_p\geq q^2$.  Reducing $st+1$ modulo $q$ and modulo $q^2$ gives
\[
2k^2 - 3k + 2 \equiv 0 \pmod{q}, \qquad
(4q+2)k^2 - (4q+3)k + 2 \equiv 0 \pmod{q^2}.
\]
The first congruence implies that $k$ is even. Multiplying the first congruence by $2q$ gives $4qk^2 - 6qk + 4q \equiv 0 \pmod{q^2}$.   Subtracting this from the second congruence yields
\[
2k^2 + (2q-3)k + 2 - 4q \equiv 0 \pmod{q^2}.  
\]
Since $2\leq k<q/3$  we have the estimate $0<2k^2+(2q-3)k+2-4q<2q^2/9+2q^2/3<q^2$, which gives a contradiction.
Therefore $T_\alpha$ cannot be of type (f).
    
    \medskip
    \textbf{Case 2: $T_\a$ is of type (i).}
    Then $q$ is even and $|\PA|=q^4(q^2-1)^2/4$.
    By Lemma \ref{lem:T_a-f-i-and-T_l-k} and the restrictions on $q$, $T_\ell$ can only be of type (n) or (s).
    Both remaining possibilities are excluded by applying  Lemma \ref{2|P|^5>|L|^4_2222}(2).
    Hence $T_\a$ cannot be of type (i).

    \medskip
    \textbf{Case 3: $T_\a$ is of type (n).}
    Then $|\PA|=q^2(q^2+\epsilon)/2$ with $\epsilon=\pm1$.
    Since $q$ is even, $T_\ell$ can only be of type (s) and $|\LA|=q^2(q+1)^2(q-1)$. If $\epsilon=-1$, then $t+1=2(q+1)(s+1)$. It follows from Lemma \ref{parameters}(3) that $t+1\leq(s+1)^2$, so $s+1\geq 2(q+1)$ and $t+1\geq 4(q+1)^2$.
    By Lemma \ref{2|P|^5>|L|^4_2222}(1),
    \[q^2(q^2-1)/2=|\PA|\geq (s+1)^2(t+1)/2\geq 8(q+1)^4,\]
    which is impossible.
    The case $\epsilon=+1$ is ruled out analogously, and we omit the details. This completes the proof.
\end{proof}

\begin{lemma}
    The subgroup $T_\a$ cannot be of type (j) or (k).
\end{lemma}
\begin{proof} We treat the two possibilities for $T_\a$ separately. 

\medskip
\textbf{Case 1: $T_\a$ is of type (j).} Since the restrictions imply that $q$ is not prime, whereas $q$ must be prime if $T_\ell$ is of type (l), (m), (o), (p), or (q), we conclude that $T_\ell$ must be of type (r). In this case
$|\PA|=\frac{q^4(q^4-1)(q^2-1)}{q_0^4(q_0^4-1)(q_0^2-1)\cdot\gcd(2,r)}$ and $|\LA|=q^3(q^4-1)$, where $q=q_0^r$ with $r$ prime. By Lemma~\ref{2|P|^5>|L|^4_2222}(2),
$q_0^{8(r-1)}/2<\frac{1}{2^{1/5}}|\PA|^{4/5}<|\LA|<q_0^{7r}$, which yields $r<9$. Thus $r\in\{2,3,5,7\}$. The case $r=2$ is excluded by Lemma~\ref{2|P|^5>|L|^4_2222}(1). We treat the remaining cases separately.

Assume $r=3$, so $q=q_0^3$. Let $d=\gcd(q_0^2+2,3)$, $A=(q_0^4+q_0^2+1)/d$, $B=q_0(q_0^4-1)/d$, and $x=q_0^8(q_0^8+q_0^4+1)d$. Then
$|\PA|=Ax$ and
$|\LA|=Bx$, so $\frac{s+1}{t+1}=\frac{A}{B}$. Since $\gcd(A,B)=1$, write $s+1=Ak$ and $t+1=Bk$ with $k>0$, whence $st+1=x/k$ and $k\mid x$. We first rule out the cases $q_0<9$ by solving equations (\ref{number_point}) and (\ref{number_line}). Thus we assume $q_0\geq 9$. By Lemma~\ref{2|P|^5>|L|^4_2222}(1), we have $Ax>(s+1)^2(t+1)/2=A^2Bk^3/2$ and hence $k<q_0^3$ for $q_0\geq 9$. Write $k=mq_0^2+aq_0+b$ with $0 \leq m,a,b<q_0$.
Since $\gcd(s+1,t+1,st+1)$ divides $2$ and $q_0$ is odd, we conclude $\gcd(q_0,k)=1$, thus we have $st+1\equiv0\pmod{q_0^4}$. As
$st+1=(Ak-1)(Bk-1)+1=ABk^2-(A+B)k+2$, letting $\alpha=q_0^4+q_0^2+1$ and $\beta=q_0(q_0^4-1)$ (so $A=\alpha/d$, $B=\beta/d$) and multiplying by $d^2$ gives
\begin{equation}\label{jjjjj}
    \alpha\beta k^2-d(\alpha+\beta)k+2d^2\equiv0\pmod{q_0^4}.
\end{equation}
Reducing \eqref{jjjjj} modulo $q_0$ yields $b \equiv 2d \pmod{q_0}$; modulo $q_0^2$ yields $a \equiv -2d \pmod{q_0}$; modulo $q_0^3$ yields $m \equiv 4d \pmod{q_0}$.  Reducing modulo \(q_0^4\) gives \(q_0\mid18d^2\). Since \(d\mid3\) and \(p\geq5\), we have \(\gcd(q_0,d)=1\), and hence \(q_0\mid18\). This is impossible, because \(q_0\) is a power of the characteristic prime \(p\geq5\).

 Assume $r=5$, so $q=q_0^5$. Let $d=\gcd(q_0^2-1,5)$, $A=q_0(q_0^8+q_0^6+q_0^4+q_0^2+1)/d$, $B=(q_0^4-1)/d$, and $x=q_0^{15}(q_0^{16}+q_0^{12}+q_0^8+q_0^4+1)d$. Then $|\PA|=Ax$, $|\LA|=Bx$, and $\gcd(A,B)=1$. As above, write $s+1=Ak$ and $t+1=Bk$ with $k>0$, whence $st+1=x/k$ and $k\mid x$. We first rule out the cases $q_0<12000$ by solving equations (\ref{number_point}) and (\ref{number_line}). Thus we assume $q_0\geq12000$. Lemma \ref{2|P|^5>|L|^4_2222}(1) yields $k<7q_0^6$. As in the case $r=3$, we have $\gcd(k,q_0)=1$, and hence $q_0^{15}\mid st+1$. Let $\alpha=q_0(q_0^8+q_0^6+q_0^4+q_0^2+1)$ and $\beta=q_0^4-1$. Since $A=\alpha/d$ and $B=\beta/d$, we obtain 
\begin{equation}\label{alpha_beta.....}
\alpha\beta k^2-d(\alpha+\beta)k+2d^2\equiv0\pmod{q_0^7}. 
\end{equation}
Reducing the left-hand side of (\ref{alpha_beta.....}) modulo $q_0$, we have $k\equiv-2d\pmod{q_0}$. Then we successively reduce this congruence modulo powers of $q_0$. Suppose that $k\equiv k_i\pmod{q_0^i}$ has been determined. Writing $k=k_i+c_iq_0^i$, substituting it into the congruence, and reducing modulo $q_0^{i+1}$ uniquely determines $c_i$ modulo $q_0$. Repeating this process gives $k\equiv d(-2202q_0^6+466q_0^5-104q_0^4+24q_0^3-6q_0^2+2q_0-2)\pmod{q_0^7}$. Set $R=q_0^7+d(-2202q_0^6+466q_0^5-104q_0^4+24q_0^3-6q_0^2+2q_0-2)$. Since $d\leq5$ and $q_0\geq12000$, we have $7q_0^6<R<q_0^7$. Since $0<k<7q_0^6<q_0^7$ and $k\equiv R\pmod{q_0^7}$, we must have $k=R$, a contradiction. Hence $r=5$ is impossible. 

Assume finally that $r=7$, so $q=q_0^7$. Let $d=\gcd(q_0^2-1,7)$, $A=q_0^3(q_0^{12}+q_0^{10}+q_0^8+q_0^6+q_0^4+q_0^2+1)/d$, $B=(q_0^4-1)/d$, and $x=q_0^{21}(q_0^{24}+q_0^{20}+q_0^{16}+q_0^{12}+q_0^8+q_0^4+1)d$. Then $|\PA|=Ax$, $|\LA|=Bx$, and $\gcd(A,B)=1$. Again, write $s+1=Ak$ and $t+1=Bk$ with $k>0$, whence $st+1=x/k$ and $k\mid x$. A direct computation excludes all cases with $q_0\leq 203$, so we assume $q_0>203$. Lemma~\ref{2|P|^5>|L|^4_2222}(1) gives $k<9q_0^9$. As above, $\gcd(k,q_0)=1$, and hence $q_0^{21}\mid st+1$. Let $\alpha=q_0^3(q_0^{12}+q_0^{10}+q_0^8+q_0^6+q_0^4+q_0^2+1)$ and $\beta=q_0^4-1$. Then $\alpha\beta k^2-d(\alpha+\beta)k+2d^2\equiv0\pmod{q_0^{10}}$. Successively reducing this congruence modulo powers of $q_0$ gives $k\equiv d(28q_0^9-14q_0^8+6q_0^7-6q_0^6+2q_0^5-2q_0^4+2q_0^3-2)\pmod{q_0^{10}}$. Since $d\leq7$ and $q_0>203$, we have $9q_0^9<d(28q_0^9-14q_0^8+6q_0^7-6q_0^6+2q_0^5-2q_0^4+2q_0^3-2)<q_0^{10}$. This contradicts $0<k<9q_0^9$. Therefore $r=7$ is also impossible.

      \medskip
    \textbf{Case 2: $T_\a$ is of type (k).} Then $q=q_0^r$ is even with $r$ prime.
    By the restrictions on $q$, $T_\ell$ can only be of type (n) or (s). We treat type (n), as the other case is similar.
    If $T_\ell$ is of type (n), then $|\LA|=q^2(q^2+\epsilon)/2<q^4$ for $\epsilon=\pm1$.
    By Lemma \ref{2|P|^5>|L|^4_2222}(2), we have $\frac{1}{2^{1/5}}q_0^{8(r-1)}<\frac{1}{2^{1/5}}|\PA|^{4/5}<|\LA|<q_0^{4r}$.
    If $r\geq 3$, then $q_0^{4r-8}<2^{1/5}$, which is impossible. Hence $r=2$.
    In particular, $|\PA|=q^2(q^2+1)(q+1)$, and this case can be excluded by the same argument as in \textbf{Case 3} of Lemma \ref{lem:type-f-i-n}.
\end{proof}

\begin{lemma}
    The subgroup $T_\a$ cannot be of type (l), (m), (o), (p), or (q).
\end{lemma}
\begin{proof}
We give the details for type (l) and the arguments for types (m), (o), (p) and (q) are analogous and are omitted. Assume that $T_\a$ is of type (l). Then $|\PA|=q^4(q^4-1)(q^2-1)/3840$, where $q=p\equiv \pm1 \pmod{8}$.  From Table \ref{tab:maximal_subgroups}, the possibilities for $T_\ell$ are types (o), (p), (q) and (r). If $T_\ell$ is of type (o), then $|\LA|=q^4(q^4-1)(q^2-1)/720$ and hence $\frac{s+1}{t+1}=\frac{|\PA|}{|\LA|}=\frac{3}{16}$. Thus we may assume that $s+1=3k$ and $t+1=16k$ for some positive integer $k$. Then $s+t=19k-2$ and $st(s+1)(t+1)=48k^2(3k-1)(16k-1)$. Note that $\gcd(19k-2,3k-1)=\gcd(19k-2,16k-1)$ divides 13. Applying Lemma \ref{parameters}(2), we obtain $19k-2 \mid 48\cdot 4\cdot 13^2=32448$, and therefore $k \in \{9,22,854\}$. For these values, we check that there are no integers $(s, t)$
that satisfy both  (\ref{number_point}) and (\ref{number_line}). 
The same calculation excludes $T_\ell$ of type (p) and (q), since the corresponding expressions for $|\LA|$ produce an analogous index computation and again lead to incompatible values. If $T_\ell$ is of type (r), then $|\mathcal{L}| = q^3(q^4-1)$. By Lemma \ref{2|P|^5>|L|^4_2222}(2), we deduce that $q\leq839$. A computer check shows that for all such $q$ there are no pairs $(s,t)$ satisfying both  (\ref{number_point}) and (\ref{number_line}), a contradiction. Therefore, $T_\a$ cannot be of type (l). This completes the proof.
\end{proof}

\begin{proof}[Proof of Theorem \ref{main}]
Let $G\le \Aut(\mc S)$ act primitively on points and lines of $\mc{S}$, and suppose that $\mathrm{Soc}(G)=T=\PSp_4(q)$ with $q\geq3$. For a point $\alpha\in\mc P$ and a line $\ell\in\mc L$, the stabilizers $G_\alpha$ and $G_\ell$ are maximal in $G$, so $T_\alpha=T\cap G_\alpha$ and $T_\ell=T\cap G_\ell$  occur in Table \ref{tab:maximal_subgroups}. If $T_\alpha$ and $T_\ell$ are of the same type, then the results of Section~\ref{section:same-type}, together with Lemma~\ref{W(3,q)-q-even}, show that $\mc S$ is the classical symplectic generalized quadrangle $W(3,q)$ with $q$ even. If $T_\alpha$ and $T_\ell$ are of different types, then the results of Section \ref{section:differnet-type}, together with Lemmas \ref{W(3,q)-q-odd} and \ref{H(3,4)}, show that, up to duality, either $\mc S\cong W(3,q)$ for $q$ odd, or $q=3$ and $\mc S\cong H(3,4)$. 
\end{proof}

\section*{Acknowledgments}
The authors thank Hanlin Zou for pointing out an error in an earlier version. This research was supported by National Key R$\&$D Program of China under grant number
2025YFA1017700. The second author was supported by the National Natural Science Foundation of China (No. 12501467) and Innovation Research Foundation of College of Science at National University of Defense Technology (202501-YJRC-LXY-02).

\section*{Data availability} This study has no associated data.

\begingroup
\sloppy
\bibliographystyle{abbrv}  
\bibliography{references}
\endgroup

\end{document}